\documentclass{article}
\usepackage[english]{babel}
\usepackage{amsthm}
\usepackage{amssymb}
\usepackage{amsmath}
\usepackage{amsfonts}
\theoremstyle{definition}

\newtheorem{theorem}{Theorem}[section]
\newtheorem{definition}{Definition}[section]
\newtheorem{example}{Example}[section]
\newtheorem{remark}{Remark}[section]
\title{The correct structures in fuzzy soft set theory}
\date{}

\begin{document}
\author{Santanu Acharjee$^1$ and Sidhartha Medhi$^2$\\
$^{1,2}$Department of Mathematics\\
Gauhati University\\
Guwahati-781014, Assam, India.\\
e-mails: $^1$sacharjee326@gmail.com,
$^2$sidharthamedhi365@gmail.com\\
Orchid: $^1$0000-0003-4932-3305,
$^2$0009-0001-6692-3647}

\maketitle
{\bf Abstract:} In 1999, Molodtsov \cite{1} developed the idea of soft set theory, proving it to be a flexible mathematical tool for dealing with uncertainty. Several researchers have extended the framework by combining it with other theories of uncertainty, such as fuzzy set theory, intuitionistic fuzzy soft set theory, rough soft set theory, and so on. These enhancements aim to increase the applicability and expressiveness of soft set theory, making it a more robust tool for dealing with complex, real-world problems characterized by uncertainty and vagueness.
The notion of fuzzy soft sets and their associated operations were introduced by Maji et al. \cite{7}. However, Molodtsov \cite{3} identified numerous incorrect results and notions of soft set theory that were introduced in the paper \cite{7}. Therefore, the derived concept of fuzzy soft sets is equally incorrect since the basic idea of soft sets in \cite{7} is flawed. Consequently, it is essential to address these incorrect notions and provide an exact and formal definition of the idea of fuzzy soft sets. This reevaluation is important to guarantee fuzzy soft set theory's theoretical stability and practical application across a range of domains. In this paper, we propose fuzzy soft set theory based on Molodtsov's correct notion of soft set theory and demonstrate a fuzzy soft set in matrix form. Additionally, we derive several significant findings on fuzzy soft sets.\\


{\bf 2020 AMS Classifications: } 03B52; 03E72;  94D05.\\

\section{Introduction}
Addressing complex issues in fields such as computer science, information technology, economics, and engineering is often challenging because traditional methods cannot adequately manage the numerous uncertainties involved. Successfully resolving such problems requires advanced mathematical techniques that effectively handle these uncertainties. A key concept in this area is mathematical uncertainty, which deals with the inherent unpredictability and variability present in many systems and events. Probability theory plays a crucial role by providing a framework to quantify and manage the likelihood of various outcomes. Furthermore, statistical methods enable mathematicians to analyze data, draw inferences, and predict future events despite the randomness and variability. To enhance our ability to address these complexities, integrating machine learning and artificial intelligence with traditional statistical methods has become increasingly important. These technologies can process vast amounts of data, identify patterns, and make predictions with greater accuracy. By combining these advanced tools, we can improve decision-making processes in various disciplines, offering more robust solutions to problems characterized by uncertainty.\\

Molodtsov \cite{1} first introduced the concept of soft sets, providing a flexible framework for managing uncertainty and vagueness. Traditional approaches such as interval mathematics, fuzzy set theory, rough set theory, and probability theory have limitations in dealing with uncertain information because they rely on specific parameterization and membership functions. In contrast, soft set theory offers a more adaptable and generalized method, allowing for the consideration of parameters without predefined membership functions. Since Molodtsov's introduction, numerous researchers have explored and extended the concept of soft sets, along with their operations and relationships. Maji \cite{7} defined operations on soft sets to examine their fundamental properties, while Chen \cite{8}, Pei, and Miao \cite{9} identified inaccuracies in some of Maji's findings and proposed new concepts and properties. Furthermore, Molodtsov \cite{3} himself pointed out errors in some definitions and operations related to soft set theory as presented by other researchers.\\

Molodtsov \cite{3} stated the following: \textit{``...many authors have introduced new operations and relations for soft sets and used these structures in various areas of mathematics and in applied science. Unfortunately, in some works, the introduction of operations and relationships for soft sets was carried out without due regard to the specifics of soft set definition.”} In \cite{3}, the correct concepts of soft operations are clearly defined. Almost all of the authors of the previously mentioned research used incorrect definitions and concepts related to soft sets \cite{3}. Molodtsov also suggested in \cite{3} that the authors of \cite{7} provided incorrect definitions for the complement, union, and intersection of soft sets. Furthermore, several researchers have since used these incorrect notions in their work. Also, in \cite{2}, Molodtsov stated the following: \textit{``The principle of constructing correct operations for soft sets is very simple. Operations should be defined through families of sets $\tau(S,A)$.”} Thus, if we want to define certain operations in soft set theory, they must be determined by the set $\tau(S, A)$ of soft set $(S, A)$, rather than by the set of parameters. In a soft set $(S, A)$, the set of parameters plays just an auxiliary role \cite{2}. Parameters are used as short names for subsets of the family $\tau(S, A)$. Molodtsov \cite{4} also commented, \textit{``Dear Colleagues, I did not write such a definition of a soft set. You have added one extra condition that the set of parameters is a subset of a fixed set.”} To define a soft set $(S, A)$, practically all authors have been regarding the set of parameters as a subset of a fixed set. There are also many inaccuracies in the theoretical foundations of soft set theory. One may refer to Acharjee and Oza \cite{4}, Acharjee and Molodtsov \cite{24}.\\ 

The theory of fuzzy sets, introduced by Zadeh \cite{5}, is highly effective for managing uncertainty. As our understanding of uncertainty deepens, the demand for more flexible and expressive frameworks increases. To enhance this capability, fuzzy set theory has been combined with soft set theory, resulting in ``fuzzy soft sets." This theoretical construct aims to provide comprehensive and adaptable solutions for scenarios where traditional binary logic or precise values fall short. Molodtsov stated in his paper \cite{3}: \textit{``...the generation of hybrid approaches combining soft sets and fuzzy sets should be carried out very carefully, taking into account differences between the concepts of equivalence.”} Thus, it is essential to understand and address these differences to ensure that the resulting hybrid models are both theoretically sound and practically applicable. \\

Maji et al. \cite{7} established the concept of fuzzy soft sets and their related operations. In \cite{7}, the authors defined the fuzzy soft set by assuming the set of parameters to be a subset of a fixed set. This is an incorrect notion according to \cite{3, 4}. They also defined the "NOT set of a set of parameters" as follows: "Let $E={e_1, e_2, \dots, e_n}$ be a set of parameters. The NOT set of $E$ is denoted by $\daleth E ={\daleth e_1, \daleth e_2, \dots, \daleth e_n}$, where $\daleth e_i = \text{not } e_i$, $\forall i$." Using this concept, they defined the complement of a fuzzy soft set. However, Molodtsov commented \cite{3} that the expression `$\daleth e_i = \text{not } e_i$, $\forall i$' looks completely mysterious. He \cite{3} stated: \textit{“When defining a soft set, there are no restrictions on the set of parameters. Different objects may play the role of parameters. They can be numbers, words, sentences, subsets—generally speaking, anything that the author chooses while introducing a soft set. Therefore, the expression, $\daleth e_i = \text{not } e_i$, $\forall i$ looks a complete mystery.”}\\

Maji et al. \cite{7} also defined the binary operations, union and intersection, for any two soft sets $(S, A)$ and $(G, B)$ defined over a universal set $X$. For these operations, they considered the set of parameters of the resultant fuzzy soft set to be $A \cup B$ for unions and $A \cap B$ for intersections. Molodtsov \cite{2} stated that the operations on soft sets should be defined through families of sets $\tau(S,A)$. In his foundational work, he clearly stated that for any soft sets $(S, A)$ and $(G, B)$, the set of parameters of the resultant soft set in binary operations should be the Cartesian product of the sets of parameters $A$ and $B$, i.e., $A \times B$. This difference in approach underscores the need for careful consideration and possible reconciliation of these methodologies to ensure consistent and accurate application of soft set theory in its hybrid forms with fuzzy theory, known as "fuzzy soft set theory."
\\

The aforementioned incorrect concepts are widely utilized by researchers in science, social science, and various branches of mathematics. Despite their extensive applications across different mathematical fields, the existing foundational ideas of fuzzy soft set theory are not correct. Therefore, it is essential to re-examine and correct these misconceptions in fuzzy soft set theory to ensure their accurate and effective use in the future. In this paper, we aim to define key concepts such as fuzzy soft sets, equal fuzzy soft sets, and equivalent fuzzy soft sets, following the correct notions introduced by Molodtsov \cite{1, 2, 3}. We seek to establish various unary and binary operations on fuzzy soft sets through their families of fuzzy sets, developing significant results in the process. Additionally, we present fuzzy soft sets in matrix form, examining the behavior of two fuzzy soft sets through their matrix representations. This study explores these operations using their matrix representations, providing a comprehensive understanding of their interactions and properties. By accurately addressing these foundational concepts, we contribute to the theoretical and practical advancements in fuzzy soft set theory, ensuring its effective application in various fields.\\

\section{Preliminaries}

A soft set can be described as a mathematical structure consisting of two key components: one is a set of parameters $A$, and the other is a universal set $X$. In this section, we explore the fundamental concepts of soft set theory along with fuzzy soft set theory and their related outcomes.\\

\begin{definition} (\cite{1, 3, 4})
A soft set over the universal set $X$ is a pair $(S,A)$ where $S$ is a mapping from the set of parameters $A$ into the set of all subsets of $X$, i.e., $S : A \to 2^X$. In fact, a soft set is a parametrized family of subsets. For a given soft set $(S,A)$, the family $\tau(S,A)$ can be defined as $\tau(S,A)=\{S(a): a \in A\}$. \\
    
\end{definition}

\begin{definition}\cite{7}
    Let $U$ be an initial universe set and $E$ be a set of parameters. Let $\mathcal{P}(U)$ denotes the set of all fuzzy sets of $U$. Let $A\subset E$. A pair $(F, A)$ is called a fuzzy soft set over $U$, where $F$ is a mapping given by $F:A\to \mathcal{P}(U)$.\\
\end{definition}

\begin{definition}\cite{7}
    For two fuzzy soft sets $(F, A)$ and $(G, B)$ over a common universe $U$, we say that $(F, A)$ is a fuzzy soft subset of $(G, B)$ if 

    (i) $A\subset B$,
    
    (ii) $\forall \varepsilon \in A$, $F(\varepsilon)$ is a fuzzy subset of $G(\varepsilon)$.

    $(F, A)$ is said to be a fuzzy soft super set of $(G, B)$, if $(G, B)$ is a fuzzy soft subset of $(F, A)$. \\
\end{definition}

\begin{definition}\cite{7}
    Two fuzzy soft sets $(F, A)$ and $(G, B)$ over a common universe $U$ are said to be fuzzy soft equal if $(F, A)$ is a fuzzy soft subset of $(G, B)$ and $(G, B)$ is a fuzzy soft subset of $(F, A)$.\\
\end{definition}

\begin{definition}\cite{7}
    The complement of a fuzzy soft set $(F, A)$ is denoted by $(F, A)^c$ and is defined by $(F, A)^c=(F^c, \daleth A)$, where $F^c : \daleth A \to \mathcal{P}(U)$ is a mapping given by $F^c(\alpha)= \;fuzzy \;complement\; of\; F(\daleth\alpha)$, $\forall\alpha \in \daleth A$.\\
\end{definition}

\begin{definition}\cite{7}
Union of two fuzzy-soft-sets of $(F, A)$ and $(G, B)$ over the common universe $U$ is the fuzzy soft-set $(H, C)$, where $C=A\cup B$ and $\forall e\in C$,

    \[H(e)=\begin{cases}
        F(e), & e\in A-B\\
        G(e), & e\in B-A\\
        F(e) \tilde{\cup} G(e), & e\in A\cap B
    \end{cases}
    \]

    and it is denoted by $(F, A) \tilde{\cup} (G, B)=(H, C)$.\\
\end{definition}

\begin{definition}\cite{7}
    Intersection of two fuzzy-soft-sets of $(F, A)$ and $(G, B)$ over the common universe $U$ is the fuzzy soft-set $(H, C)$, where $C=A\cap B$ and $\forall e\in C$, $H(e)=F(e) \;or\; G(e)$, (as both are same set). We write $(F, A) \tilde{\cap} (G, B)=(H, C)$.\\
\end{definition}

\begin{definition}\cite{5}
A fuzzy set $A$ in $X$ is characterized by a membership (characteristic) function $\mu_A (x)$ which associates with each point in $X$ a real number in the interval [0, 1],  with the value of $\mu_A(x)$ at $x$ representing the ``grade of membership" of $x$ in $A$. \\
\end{definition}

\begin{definition}\cite{5}
   A fuzzy set is empty if and only if its membership function is identically zero on $X$.\\ 
\end{definition}

\begin{definition}\cite{5}
Two fuzzy sets $A$ and $B$ are equal, written as $A = B$, if and only if $\mu_A(x) = \mu_B(x)$ for all $x$ in $X$.\\
\end{definition}

\begin{definition}\cite{6}
    The fuzzy set $A$ is included in the fuzzy set $B$ (or, equivalently, $A$ is a subset of $B$, or $A$ is smaller than or equal to $B$) if for every $x\in X$, $\mu_A(x) \leq \mu_B(x)$. It is denoted by the  symbols $A\subseteq B$.\\
\end{definition}

\begin{definition}\cite{6}
     A fuzzy set $A$ is said to be proper subset of the fuzzy set $B$ denoted by $A\subset B$, when $A$ is subset of $B$ and $A \neq B$, i.e., there exists at least one $x\in X$ such that $\mu_A(x) < \mu_B(x)$.\\
\end{definition}

\begin{definition}\cite{6}
   The complement of a fuzzy set $A$ is denoted by $\overline{A}$ and is defined by $\mu_{\overline{A}}(x)=1-\mu_A(x)$.\\
\end{definition}

\begin{definition}\cite{5}
    The union of two fuzzy sets $A$ and $B$ with respective membership functions $\mu_A(x)$ and $\mu_B(x)$ is a fuzzy set $C$, written as $C=A\cup B$, whose membership function is related to those of $A$ and $B$ by 
\begin{center}
    $\mu_C(x) = Max[\mu_A(x), \mu_B(x)]$, $x\in X$.\\
\end{center}
\end{definition}

\begin{definition}\cite{5}
The intersection of two fuzzy sets $A$ and $B$ with respective membership functions $f_A(x)$ and $f_B(x)$ is a fuzzy set $C$, written as $C=A\cap B$, whose membership function is related to those of $A$ and $B$ by 
\begin{center}
    $\mu_C(x) = Min[\mu_A(x), \mu_B(x)]$, $x\in X$.\\
\end{center}
\end{definition}

In this study, we refer to $\tilde{\phi}$ as empty fuzzy set and $\tilde{X}$ as a fuzzy set such that $\mu_{\tilde{X}}(x)=1$, for every $x\in X$.

\section{Main Results}
In the past few years, many researchers have introduced the concept of fuzzy soft sets and their related operations \cite{7}. In \cite{2, 3}, Molodtsov showed the incorrectness in definitions and related operations of soft set and fuzzy soft set. Now, we try to develop the correct definition and operation of fuzzy soft sets so that future researchers can carry this out in their respective papers or projects in the field of fuzzy soft sets. In this section, we consider $(S,A)$ and $(G,B)$ as fuzzy soft sets defined over a universal set $X$ and $A$ and $B$ as the set of parameters. \\

\begin{definition}
    A fuzzy soft set is a pair $(S,A)$ defined over a universal set $X$ if and only if $S$ is a mapping from the set of parameters $A$ into the set of all fuzzy subsets of $X$, i.e., $S : A\to P(X)$. So, in a fuzzy soft set $(S, A)$, each $S(a)$, $\forall a \in A$, is a fuzzy set. For a given fuzzy soft set (S,A), the family $\tau (S,A)$ can be defined as the set $\tau (S,A)=\{S(a), a\in A \}$.

    Here, $P(X)$ denotes the set of all fuzzy sets of $X$.\\
\end{definition}
\begin{example}
    Suppose $X=\{a, b, c\}$ and $A=\{x, y, z\}$ are universal set and set of parameters respectively. We define a fuzzy soft set $(S,A)$ as follows: $S(x)=\{\frac{b}{0.3}, \frac{c}{0.7}\}$, $S(y)=\{\frac{c}{0.1}\}$ and  $S(z)=\{\frac{a}{0.5}\}$. Therefore, $(S,A)=\{(x, \{\frac{b}{0.3}, \frac{c}{0.7}\}), (y, \{\frac{c}{0.1}\}), (z,\{\frac{a}{0.5}\})\}$ and  $\tau (S,A)=\{ \{\frac{b}{0.3}, \frac{c}{0.7}\}), \{\frac{c}{0.1}\}, \{\frac{a}{0.5}\}\}$.\\
\end{example}
\begin{definition}
       Two fuzzy soft sets $(S,A)$ and $(G,B)$ defined over a universal set $X$, are said to be equal if and only if $S=G$ and $A=B$, i.e., $S(a)=G(a)$, $\forall a\in A$. In this case it is denoted by $(S,A)=(G,B)$.\\
\end{definition}
\begin{definition}
Two fuzzy soft sets $(S,A)$ and $(G,B)$ defined over a universal set $X$, are said to be equivalent if and only if $\tau (S,A)=\tau (G,B)$. It is denoted by $(S,A) \cong (G,B)$.\\
\end{definition}

\begin{example}
    Consider a fuzzy soft set $(G,B)$ defined over a universal set $X=\{a, b, c\}$. We define this fuzzy soft set as follows: \\
    
    ~~~~~~~~~~~~~~~$(G,B)=\{(m, \{\frac{b}{0.3}, \frac{c}{0.7}\}), (n, \{\frac{a}{0.5}\}), (o,\{\frac{c}{0.1}\})\}$\\

    Here,  $\tau (G,B)=\{ \{\frac{b}{0.3}, \frac{c}{0.7}\}), \{\frac{a}{0.5}\}, \{\frac{c}{0.1}\}\}$. So, the fuzzy soft sets $(S,A)$ of example 3.1 and $(G,B)$ are equivalent. So, we can write as $\tau(S,A)= \tau(G,B)$.
\end{example}

\begin{definition}
    A fuzzy soft set $(S,A)$ internally approximates a fuzzy soft set $(G,B)$ defined over a universal set $X$, denoted by $(S,A)\subseteq (G,B)$ if for any $b\in B$ such that $G(b)$ is not an empty fuzzy set, i.e., $G(b)\neq \tilde{\phi}$, there exists $ a \in A$ for which $\tilde{\phi} \neq S(a) \subseteq G(b)$.
    
\end{definition}

\begin{example}
    We consider two fuzzy soft sets $(S,A)$ and $(G,B)$ defined over a universal set $X=\{a, b, c\}$ as follows: $(S,A)=\{(x, \{\frac{b}{0.3}, \frac{c}{0.7}\}), (y, \{\frac{c}{0.1}\}), (z,\{\frac{a}{0.5}\})\}$ and $(G,B)=\{(m, \{\frac{a}{0.5}\}), (n, \{\frac{c}{0.2}\}), (o,\{\frac{c}{0.4}\})\}$. Then, $(S,A)\subseteq (G,B)$.
\end{example}
\begin{definition}
      A fuzzy soft set $(S,A)$ externally  approximates a fuzzy soft set $(G,B)$ defined over a universal set $X$, denoted by $(S,A)\supseteq (G,B)$ if for any $b\in B$ such that $G(b)\neq \tilde{X}$, there exists $a \in A$ for which $\tilde{X} \neq S(a) \supseteq G(b)$.
\end{definition}

\begin{example}
    In example 4.3, the fuzzy soft set $(S,A)$ externally approximates the fuzzy soft set $(G,B)$.
\end{example}

\begin{definition}
    A fuzzy soft set $(S,A)$ internally strictly approximates a fuzzy soft set $(G,B)$ defined over a universal set $X$ and it is denoted by $(S,A)\subset (G,B)$ if $(S,A)\subseteq (G,B)$ but the relation $(G,B)\subseteq (S,A)$ has no place.
\end{definition}

\begin{example}
    In example 3.3, $(S,A)\subseteq (G,B)$ but
    $(G,B)\nsubseteq (S,A)$. Therefore, we can say that $(S,A)\subset (G,B)$.
\end{example}

\begin{definition}
     A fuzzy soft set $(S,A)$ externally strictly approximates a fuzzy soft set $(G,B)$ defined over a universal set $X$ and it is denoted by $(S,A)\supset (G,B)$ if $(S,A)\supseteq (G,B)$ but the relation $(G,B)\supseteq (S,A)$ has no place.
\end{definition}

\begin{example}
    In example 3.3, $(S,A)\supseteq (G,B)$ but
    $(G,B)\nsupseteq (S,A)$. Therefore, we can write that $(S,A)\supset (G,B)$.
\end{example}

\begin{definition}
    Let $(S, A)$ and $(G, B)$ are two fuzzy soft sets defined over a universal set $X$. Then the fuzzy soft set $(S,A)$ is internally equivalent to the fuzzy soft set $(G,B)$ if $(S,A)\subseteq (G,B)$ and $(G,B)\subseteq (S,A)$. It is denoted by $(S,A)\stackrel{\subset}{\approx}(G,B)$.
\end{definition}

\begin{example}
    Consider two fuzzy soft sets $(S,A)$ and $(G,B)$ defined over a universal set $X$, where $X=\{a, b, c\}$, $A=\{x, y, z\}$ and $B=\{m, n\}$. Now we define these two fuzzy soft sets as,  
    $(S,A)=\{(x, \{\frac{a}{0.2}\}), (y, \{\frac{a}{0.4}\}), (z,\{\frac{a}{0.4}, \frac{b}{0.5}\})\}$ and $(G,B)=\{(m, \{\frac{a}{0.3}, \frac{b}{0.5}\}), (n, \{\frac{a}{0.2}\})\}$. From the above, we can clearly find that $(S,A)\subseteq (G,B)$ and $(G,B)\subseteq (S,A)$. This implies that $(S,A)\stackrel{\subset}{\approx}(G,B)$.

\end{example}

\begin{definition}
     Consider two fuzzy soft sets $(S, A)$ and $(G, B)$ defined over a universal set $X$. Then the fuzzy soft set $(S,A)$ is externally equivalent to the fuzzy soft set $(G,B)$ if $(S,A)\supseteq (G,B)$ and $(G,B)\supseteq (S,A)$. It is denoted by $(S,A)\stackrel{\supset}{\approx}(G,B)$.
\end{definition}

\begin{example}
      Let $(S, A)$ and $(G, B)$ be two fuzzy soft sets defined over a universal set $X$. We consider $X=\{a, b, c\}$, $A=\{x, y, z\}$ and $B=\{m, n, o\}$ and fuzzy soft sets $(S,A)=\{(x, \{\frac{a}{0.4}\}), (y, \{\frac{a}{0.5}, \frac{b}{0.2}, \frac{c}{0.3}\}), (z,\{\frac{a}{0.4}, \frac{b}{0.2}\})\}$ and $(G,B)=\{(m, \{\frac{a}{0.3}\}), (n, \{\frac{a}{0.5}, \frac{b}{0.2}, \frac{c}{0.3}\}), (o, \{\frac{b}{0.1}, \frac{c}{0.2})\}$.

    From the above, we can clearly find that $(S,A)\supseteq (G,B)$ and $(G,B)\supseteq (S,A)$. It implies that $(S,A)\stackrel{\supset}{\approx}(G,B)$.
\end{example}

\begin{definition}

    Let $(S, A)$ and $(G, B)$ are two fuzzy soft sets defined over a universal set $X$. Then $(S, A)$ is said to be weakly equivalent to $(G, B)$ if $(S,A)\stackrel{\subset}{\approx}(G,B)$ and $(S,A)\stackrel{\supset}{\approx}(G,B)$. It is denoted by $(S,A)\approx (G,B)$ 
\end{definition}

\begin{example}
    The following two fuzzy soft sets $(S,A)$ and $(G,B)$ defined over a universal set $X=\{a, b, c, d\}$ are weakly equivalent, where the sets of parameters are $A=\{x, y, z\}$ and $B=\{m, n, o\}$ and $(S,A)=\{(x, \{\frac{a}{0.4}\}), (y, \{\frac{a}{0.5}, \frac{b}{0.2}, \frac{c}{0.3}\}), $ $(z, \{\frac{a}{0.4}, \frac{b}{0.2}\})\}$ and $(G,B)=\{(m, \{\frac{a}{0.4}\}), (n, \{\frac{a}{0.5}, \frac{b}{0.2}, \frac{c}{0.3}\}), (o, \{\frac{a}{0.4}, \frac{c}{0.2}\})\}$.
\end{example}
\begin{definition}
    For a fuzzy soft set $(S, A)$ defined over a universal set $X$, the minimal and maximal on inclusion of sets of the family $\tau(S,A)$ are defined as:\\

$MIN(\tau(S,A))=\{C\in \tau(S,A)~|~C\neq \tilde{\phi}, \nexists~  D \in \tau(S,A) : D\subset C \;and\; D \neq \tilde{\phi} \}$

$MAX(\tau(S,A))=\{C\in \tau(S,A)~|~C\neq \tilde{X}, \nexists~  D \in \tau(S,A) : D\supset C \;and\; D \neq \tilde{X} \}$.

The elements of the set $MIN(\tau(S,A))$ are called minimal element of the family $\tau(S,A)$ and the elements of the set $MAX(\tau(S,A))$ are referred as the maximal element of family $\tau(S,A)$.
\end{definition}

\begin{example}
    Consider a fuzzy soft set $(S, A)$ defined over a universal set $X$, where the parameter set $A=\{x, y, z\}$, and the universal set $X=\{a, b, c\}$.  We define the fuzzy soft set $(S,A)$ as follows:\\

    $S(x)=\{\frac{a}{0.5}\}$, $S(y)=\{\frac{a}{0.6}, \frac{b}{0.2}, \frac{c}{0.3}\}$ and $S(z)=\{\frac{a}{0.7}, \frac{b}{0.2}\}$ \\

    Thus, $\tau(S, A)=\{\{\frac{a}{0.5}\}, \{\frac{a}{0.6}, \frac{b}{0.2}, \frac{c}{0.3}\}, \{\frac{a}{0.7}, \frac{b}{0.2}\}\}$\\

    Now, from the family of sets $\tau(S, A)$, we can clearly get that 
    $MIN(\tau(S,A))=\{\{\frac{a}{0.5}\}\}$ and $MAX(\tau(S,A))=\{\{\frac{a}{0.6}, \frac{b}{0.2}, \frac{c}{0.3}\}, \{\frac{a}{0.7}, \frac{b}{0.2}\}\}$. 
\end{example}
\begin{definition}
    Let $(S, A)$ be a fuzzy soft set defined over a universal set $X$. Then, the complement of $(S, A)$ is denoted by $C(S, A)$ and defined as $CS(a)= \overline {S(a)}$, $\forall a\in A$.
\end{definition}

\begin{example}
    Let $(S, A)$ be a fuzzy soft set defined over a universal set $X$ where  $X=\{x, y, z\}$ and $A=\{a, b\}$. We define the fuzzy soft set $(S, A)$ as follows:

    $ S(a)=\{\frac{x}{0.2}, \frac{z}{0.8}\}$ and $S(b)=\{\frac{x}{0.7}, \frac{y}{1}\}$. So, $\overline {S(a)}=\{\frac{x}{0.8},\frac{y}{1}, \frac{z}{0.2}\}$ and $\overline {S(b)}=\{\frac{x}{0.3},\frac{y}{0}, \frac{z}{1}\}=\{\frac{x}{0.3}, \frac{z}{1}\}$. Therefore, $C(S,A)=\{(a,\{\frac{x}{0.8},\frac{y}{1}, \frac{z}{0.2}\}), (b,\{\frac{x}{0.3}, \frac{z}{1}\})\}$.
    
\end{example}

\begin{definition}
    Let $(S, A)$ and $(G,B)$ be two fuzzy soft sets defined over a universal set $X$. Then, the binary operation union of these fuzzy soft sets is again a fuzzy soft set $(U, A \times B)$ over the same universal set $X$, i.e., $(S,A) \cup (G,B)=(U, A\times B)$ and it is defined by the following:

    $U(a, b)=S(a)\cup G(b)$ , $\forall (a ,b) \in A\times B$.

    Here, the membership value of fuzzy union $S(a)\cup G(b)$ is defined as

    $\mu_{U(a,b)}(x) = \mu_{S(a)\cup G(b)}(x) = Max\{\mu_{S(a)}(x), \mu_{G(b)}(x)\}$, $\forall x\in X$.
\end{definition}

\begin{definition}
     Let $(S, A)$ and $(G,B)$ be two fuzzy soft sets defined over a universal set $X$. Then, the binary operation intersection of these fuzzy soft sets is again a fuzzy soft set $(I, A \times B)$ over the same universal set X i.e., $(S,A) \cap (G,B)=(I, A\times B)$ and it is defined by the following:

    $I(a, b)=S(a)\cap G(b)$ , $\forall (a ,b) \in A\times B$.

    Here, the membership value of fuzzy intersection $S(a)\cap G(b)$ defined as

    $\mu_{I(a,b)}(x) = \mu_{S(a)\cap G(b)}(x) = Min\{\mu_{S(a)}(x), \mu_{G(b)}(x)\}$, $\forall x\in X$.
\end{definition}
\begin{example}
    Let us consider two fuzzy soft sets $(S, A)$ and $(G, B)$ defined over a  universal set $X=\{a, b, c\}$. We also consider the set of parameters as $A=\{x, y\}$ and $B=\{p, q\}$ and we define these two fuzzy soft sets as given below : \\

    $S(x)=\{\frac{a}{0.4}, \frac{b}{0.8}\}$, $S(y)=\{\frac{c}{0.7}\}$, $G(p)=\{\frac{b}{0.3}, \frac{c}{0.5}\}$ and $G(q)=\{\frac{a}{0.6}, \frac{c}{1}\}$\\

    Then, $U(x, p)= S(x) \cup G(p)= \{\frac{a}{0.4}, \frac{b}{0.8}\} \cup \{\frac{b}{0.3}, \frac{c}{0.5}\} = \{\frac{a}{0.4}, \frac{b}{0.8}, \frac{c}{0.5}\}$\\

     $U(x, q)= S(x) \cup G(q)= \{\frac{a}{0.4}, \frac{b}{0.8}\} \cup \{\frac{a}{0.6}, \frac{c}{1}\} = \{\frac{a}{0.6}, \frac{b}{0.8}, \frac{c}{1}\}$,\\

     $U(y, p)= S(y) \cup G(p)= \{ \frac{c}{0.7}\} \cup \{\frac{b}{0.3}, \frac{c}{0.5}\} = \{\frac{b}{0.3}, \frac{c}{0.7}\}$\\

      $U(y, q)= S(y) \cup G(q)= \{\frac{c}{0.7}\} \cup \{\frac{a}{0.6}, \frac{c}{1}\} = \{\frac{a}{0.6}, \frac{c}{1}\}$,\\

      Therefore, $(U, A\times B)=\{((x, p), \{\frac{a}{0.4}, \frac{b}{0.8}, \frac{c}{0.5}\}), ((x, q), \{\frac{a}{0.6}, \frac{b}{0.8}, \frac{c}{1}\}),\\ ((y,p), \{\frac{b}{0.3}, \frac{c}{0.7}\}), ((y, q), \{\frac{a}{0.6}, \frac{c}{1}\})\}$.\\

      Similarly we can find the intersection of the given fuzzy soft sets $(S, A)$ and $(G, B)$ as,

      $(I, A\times B)=\{((x, p), \{ \frac{b}{0.3}\}), ((x, q), \{\frac{a}{0.4}\}), ((y,p), \{ \frac{c}{0.5}\}), ((y, q), \{ \frac{c}{0.7}\})\}$.
\end{example}

\begin{definition}
    Consider two fuzzy soft sets  $(S, A)$ and $(G, B)$ defined over a  universal set $X$. Then the product of these two fuzzy soft sets is again a fuzzy soft set, say $(P, A\times B) $, i.e., $(P, A\times B)=(S,A) \times (G,B)$ and it is defined as $P(a, b)= S(a).G(b)$, $\forall (a, b) \in A\times B$.

    Here, $S(a).G(b)$ is fuzzy algebraic product of the fuzzy sets $S(a)$ and $G(b)$ and $\mu_{S(a).G(b)}(x)= {\mu_{S(a)}(x)}.{\mu_{G(b)}(x)}$, $\forall x \in X$.
\end{definition}

\begin{definition}
      The sum of two fuzzy soft sets $(S, A)$ and $(G, B)$ defined over a universal set $X$, is a fuzzy soft set, say $(H, A\times B) $, i.e., $(H, A\times B)=(S,A) + (G,B)$ and it is defined as $H(a, b)= S(a) \hat{+} G(b)$, $\forall (a, b) \in A\times B$.

    Here, $S(a)\hat{+} G(b)$ is the fuzzy algebraic sum of the fuzzy sets $S(a)$ and $G(b)$ and $\mu_{S(a)\hat{+} G(b)}(x)= {\mu_{S(a)}(x)}+{\mu_{G(b)}(x)}- {\mu_{S(a)}(x)}.{\mu_{G(b)}(x)}$, $\forall x \in X$.
\end{definition}

\begin{theorem}
    If two fuzzy soft sets $(S, A)$ and $(G, B)$ defined over a universal set $X$, are weakly equivalent, then

(i) $MIN(\tau(S,A))=MIN(\tau(G,B))$,

(ii)$MAX(\tau(S,A))=MAX(\tau(G,B))$.
\end{theorem}
        
 \begin{proof}
  Let $(S, A)$ and $(G, B)$ be two weakly equivalent fuzzy soft sets defined over a universal set $X$. So, $(S, A)\stackrel{\subset}{\approx}(G, B)$ and $(S, A)\stackrel{\supset}{\approx}(G, B)$. Also we know that, if $(S, A)\stackrel{\subset}{\approx}(G, B)$ then $(S, A) \subseteq (G, B)$ and $(G, B)\subseteq (S, A)$. Again, if $(S, A)\stackrel{\supset}{\approx}(G, B)$ then $(S, A) \supseteq (G, B)$ and $(G, B)\supseteq (S, A)$.\\

 (i) To prove the theorem, we first consider a fuzzy set $S(a_i)\in MIN(\tau(S,A))$. Then $S(a_i)$ is a non-empty fuzzy set, and there does not exist any non-empty fuzzy set $S(a_j) \in \tau(S,A)$ such that  $S(a_j) \subset S(a_i)$. Given that $(G, B)\subseteq (S, A)$. Therefore, for the fuzzy set $S(a_i)$, there exists a non-empty fuzzy set $G(b_k) \in \tau(G,B)$ such that $G(b_k) \subseteq S(a_i)$. Also as $(S,A) \subseteq (G,B)$. So, for the fuzzy set $G(b_k)$, there exists a non-empty fuzzy set $S(a_j) \in \tau(S,A)$ such that $S(a_j) \subseteq G(b_k)$.
Hence, we get $S(a_j)\subseteq G(b_k)\subseteq S(a_i)$. However, since $S(a_i)\in MIN(\tau(S,A))$, we must conclude that $S(a_j)=S(a_i)$. Hence $ S(a_i)=G(b_k)$. This implies that $S(a_i)\in \tau(G,B)$.\\

Now, we prove $G(b_k)\in MIN(\tau(G,B))$. For this, we first consider $G(b_k)\notin MIN(\tau(G,B))$. Then, there exists a fuzzy set $G(b_n)(\neq \tilde{\phi})\in \tau(G,B)$, such that $ G(b_n) \subset G(b_k)$. Since, $(S,A)\subseteq (G,B)$, then for the fuzzy set $G(b_n)$, there exists $S(a_k) (\neq \tilde{\phi})\in \tau(S,A)$, such that $S(a_k)\subseteq G(b_n)$.
Thus, we get $S(a_k)\subseteq G(b_n) \subset G(b_k) = S(a_i)$. This implies that $S(a_i) \notin MIN(\tau(S,A))$, which is a contradiction. Hence, $G(b_k)=S(a_i)\in MIN(\tau(G,B))$. Hence, $MIN(\tau(S,A)) \subseteq MIN(\tau(G,B))$. By the similar process, we can prove that $MIN(\tau(G,B)) \subseteq MIN(\tau(S,A))$. Hence, $MIN(\tau(S,A))=MIN(\tau(G,B))$.\\

(ii)
Let $G(b_i) \in MAX(\tau(G, B))$. Then, there does not exist any $G(b_j) (\neq \Tilde{X}) \in \tau (G, B)$ such that $G(b_j)\supset G(b_i)$. Since $(S, A)\supseteq (G, B)$, so for the fuzzy set $G(b_i)$, there exists a fuzzy set $S(a_k) (\neq \Tilde{X}) \in \tau(S, A)$ such that $S(a_k) \supseteq G(b_i)$. Also, given that $(G, B) \supseteq (S, A)$. So, for the fuzzy set $S(a_k)$, there exists a fuzzy set $G(b_j)(\neq \Tilde{X}) \in \tau(G, B)$ such that $ G(b_j) \supseteq S(a_k)$. Thus, we get $G(b_j)\supseteq S(a_k)\supseteq G(b_i)$. But, $G(b_i)\in MAX(\tau(G, B))$, thus we must get  $G(a_i)=G(b_j)$. Hence, $S(a_k)=G(b_i)$. This implies that $G(b_i)\in \tau(S, A)$.\\

Now, we have to prove that $S(a_k)\in MAX(\tau(S,A))$. For this, we first consider $S(a_k)\notin MAX(\tau(S,A))$. Then, there exists a fuzzy set $S(a_j)(\neq \Tilde{X})\in \tau(S, A)$, such that $S(a_j) \supset S(a_k)$. Since, $(G,B)\supseteq (S,A)$, so for the fuzzy set $S(a_j)$, there exists $G(b_k) (\neq \Tilde{X})\in \tau(G, B)$, such that $ G(b_k)\supseteq S(a_j)$.
Then, we get $G(b_k)\supseteq S(a_j) \supset S(a_k) = G(b_i)$, which contradicts our assumption $G(b_i) \in MAX(\tau(G,B))$. Hence, $G(b_i)=S(a_k)\in MAX(\tau(S, A))$. Hence, $MAX(\tau(G,B)) \subseteq MAX(\tau(S,A))$. Using a similar procedure, we can prove that $MAX(\tau(S,A)) \subseteq MAX(\tau(G,B))$. Hence, $MAX(\tau(S,A))=MAX(\tau(G,B))$.\\

\end{proof}
 \begin{theorem}
     Let $(S, A)$ be a fuzzy soft set given defined over a universal set $X$ and $C(S,A)$ be its complement of $(S, A)$. Then for $a\in A$, the following two properties are occur :\\

     (i) $S(a) \in MAX(\tau(S,A))$ if and only if   $\overline{S(a)} \in MIN(\tau (C(S,A)))$,

     (ii) $S(a) \in MIN(\tau(S,A))$ if and only if   $\overline{S(a)} \in MAX(\tau (C(S,A)))$.
 \end{theorem}

\begin{proof}

     We know that for any two fuzzy sets $\mathcal{A}$ and $\mathcal{B}$ defined over $X$, the following conditions are hold:

     (a) $\overline{\overline{\mathcal{A}}}= \mathcal{A}$,

     (b) $\mathcal{A} \subset \mathcal{B}$ if and only if $\overline{\mathcal{A}} \supset \overline{\mathcal{B}}$.\\

Using these results, we proceed for our proofs.\\

     (i) First we prove that if $S(a) \in MAX(\tau(S,A))$ then $\overline{S(a)} \in MIN(\tau (C(S,A)))$,  for $a\in A$.\\ 
     
     For this, let us assume that $S(a) \in MAX(\tau(S,A))$. Then, for $b\in A$, $\nexists ~S(b) \in \tau(S,A)$ such that $S(b)\supset S(a)$ and $S(b) \neq \Tilde{X}$. This implies $\nexists ~~\overline{S(b)} \in \tau(C(S,A))$ such that $\overline{S(b)}\subset \overline{S(a)}$ and $\overline{S(b)} \neq \Tilde{\phi}$. Hence, $\overline{S(a)} \in MIN(\tau (C(S,A)))$.\\

     Conversely, let $\overline{S(a)} \in MIN(\tau (C(S,A)))$, for $a\in A$. Then for $b\in A$, $\nexists ~~\overline{S(b)} \in \tau(C(S,A))$ such that $\overline{S(b)}\subset \overline{S(a)}$ and $\overline{S(b)} \neq \Tilde{\phi}$. So, $\nexists ~S(b) \in \tau(S,A)$ such that $S(b)\supset S(a)$ and $S(b) \neq \Tilde{X} $. Hence, $S(a) \in MAX(\tau(S,A))$.\\

     (ii) We first prove that if $S(a) \in MIN(\tau(S,A))$, then $\overline{S(a)} \in MAX(\tau (C(S,A)))$, for $a\in A$. Let $S(a) \in MIN(\tau (S,A))$. Then for $b\in A$, $\nexists ~~S(b) \in \tau(S,A)$ such that $S(b)\subset S(a)$ and $S(b) \neq \Tilde{\phi}$. So, $\nexists ~\overline{S(b)} \in \tau(C(S,A))$ such that $\overline{S(b)}\supset \overline{S(a)}$ and $\overline{S(b)} \neq \Tilde{X} $. Therefore, $\overline{S(a)} \in MAX(\tau(C(S,A)))$. \\

     Conversely, let $\overline{S(a)} \in MAX(\tau(C(S,A)))$, for $a\in A$.  Then, for $b\in A$, $\nexists ~\overline{S(b)} \in \tau(C(S,A))$ such that $\overline{S(b)}\supset \overline{S(a)}$ and $ \overline{S(b)} \neq \Tilde{X}$. So, $\nexists ~~S(b) \in \tau(S,A)$ such that $S(b)\subset S(a)$ and $S(b) \neq \Tilde{\phi}$. Hence, $S(a) \in MIN(\tau (S,A))$.
        
\end{proof}
\begin{definition}
    A fuzzy soft set $(S,A)$ is called fuzzy empty soft set if $\tau (S,A)= \phi$ i.e., the set of parameter is empty and it is denoted by $(-,\phi)$.
\end{definition}

\begin{definition}
     For a non-empty set of parameter $A$, the fuzzy soft set $(S,A)$ given over a universal set $X$ is called a fuzzy null soft set if $\tau(S,A)=\{\tilde{\phi} \}$ and it is denoted by the symbol $(\phi , -)$.
\end{definition}

\begin{definition}
     For a non-empty set of parameter $A$, the fuzzy soft set $(S,A)$ given over a universal set $X$ is called a fuzzy absolute soft set if $\tau(S,A)=\{\tilde{X} \}$ and it is denoted by the symbol $(X , -)$.
\end{definition}

\begin{theorem}
    Let $(S,A)$ be a fuzzy soft set defined over a universal set $X$, then $C(C(S,A))=(S,A)$.
    
\end{theorem}

\begin{proof}

 We know that for a fuzzy set $\mathcal{A}$ defined over $X$, $\overline{\overline{\mathcal{A}}}=\mathcal{A}$ holds.

 Now, consider the fuzzy soft set $(S,A)=\{(a, S(a)):a \in A\}$, where $S(a)$ is a fuzzy set. Then, the complement of $(S,A)$ is $C(S,A)=\{(a, \overline{S(a)}): a \in A\}$. Again taking complement of the fuzzy soft set $C(S,A)$, we get $C(C(S,A))=\{(a, \overline{\overline{S(a)}}): a \in A\}=\{(a, S(a)): a \in A\}=(S,A)$. Therefore, $C(C(S,A))=(S,A)$.\\
\end{proof}

Before we proceed to prove the following theorem, it is essential to revisit some concepts related to the set of parameters involving in soft sets, as outlined by Molodtsov in \cite{3} . Molodtsov \cite{3} emphasized that within soft set theory, parameters are primarily used to identify a specific subset. Therefore, the role of parameters in defining soft sets as well as fuzzy soft set is merely auxiliary. Hence, for the fuzzy soft sets $(S,A)$, $(G,B)$ and $(H,C)$; the sets of parameters $A\times B$ and $B\times A$ are indistinguishable. Similarly the sets of parameters $A\times B \times C$,  $(A\times B) \times C$ and $A\times (B \times C)$ are indistinguishable.\\

\begin{theorem}
    Let $(S,A)$, $(G,B)$ and $(H,C)$ be three fuzzy soft sets defined over a universal set $X$. Then, the following laws hold:

(a) Commutative law:

    ~~~(i) $(S,A) \cap (G,B) = (G,B) \cap (S,A)$

    ~~~(ii) $(S,A) \cup (G,B) = (G,B) \cup (S,A)$

    ~~~(iii) $(S,A) \times (G,B) = (G,B) \times (S,A)$

    ~~~(iv) $(S,A) + (G,B) = (G,B) + (S,A)$

    (b) Associative law:

    ~~~(i) $\{(S,A) \cap (G,B)\}\cap (H,C) = (S,A) \cap \{(G,B)\cap (H,C)\}$

    ~~~(ii) $\{(S,A) \cup (G,B)\}\cup (H,C) = (S,A) \cup \{(G,B)\cup (H,C)\}$

    ~~~(iii) $\{(S,A) \times (G,B)\} \times (H,C) = (S,A) \times \{(G,B) \times (H,C)\}$

    ~~~(iv) $\{(S,A) + (G,B)\} + (H,C) = (S,A) + \{(G,B) + (H,C)\}$

    (c) De Morgan's law:

    ~~~(i) $C((S,A) \cap (G,B)) = C(S,A) \cup C(G,B)$

    ~~~(ii) $C((S,A) \cup (G,B)) = C(S,A) \cap C(G,B)$

    (d) Distributive law:

    ~~~(i) $(S,A)\cup \{(G,B) \cap (H,C)\} = \{(S,A) \cup (G, B)\} \cap \{(S,A) \cup (H,C)\}$

    ~~~(ii) $(S,A)\cap \{(G,B) \cup (H,C)\} = \{(S,A) \cap (G,B)\} \cup \{(S,A) \cap (H,C)\}$

\end{theorem}
\begin{proof}

 Here, we only prove the first part of each property. To prove these properties, we consider the following fuzzy soft sets $(S,A)=\{(a, S(a)): a\in A\}$, $(G,B)=\{(b, G(b)): b\in B\}$ and  $(H,C)=\{(c, H(c)): c\in C\}$ defined over a universal set $X$.\\

(a) (i) Let us consider two fuzzy soft sets $(S,A) \cap (G,B)=(I_1, A\times B)$ and $(G,B) \cap (S,A)=(I_2, B\times A)$. Then for any $a\in A$ and $b\in B$, we have $I_1(a,b)=S(a) \cap G(b)$ and $I_2(b,a)=G(b) \cap S(a)$. Since $S(a)$ and $G(b)$ are fuzzy sets, so we get $S(a) \cap G(b)=G(b) \cap S(a)$. Also the set of parameters $A\times B$ and $B\times A$ are indistinguishable. Hence from above, we get $I_1(a,b)= I_2(b,a)$, $\forall a \in A $ and $b\in B$. Therefore, $(S,A) \cap (G,B) = (G,B) \cap (S,A)$.\\

(b) (i) For any $a\in A$, $b\in B$ and $c\in C$, the fuzzy sets $S(a)\in \tau (S,A)$, $G(b)\in \tau (G,B)$ and $H(c)\in \tau (H,C)$. Let $\{(S,A) \cap (G,B)\}\cap (H,C)= (P, (A\times B)\times C)$ be the fuzzy soft set defined by $P((a,b),c)= \{S(a)\cap G(b)\}\cap H(c)$. Again let, $(S,A) \cap \{(G,B)\cap (H,C)\}= (Q, A\times (B\times C))$ be the fuzzy soft set defined by $Q(a,(b,c))= S(a)\cap \{G(b)\cap H(c)\}$.

Since the set of parameters $A\times B \times C$,  $(A\times B) \times C$ and $A\times (B \times C)$ are indistinguishable, so $((a,b),c)\in (A\times B)\times C$ and $(a,(b,c)) \in A\times(B\times C)$ are indistinguishable. Also the associative property hold for three fuzzy sets $S(a)$, $G(b)$ and $H(c)$. Therefore, $\{S(a)\cap G(b)\}\cap H(c)=S(a)\cap \{G(b)\cap H(c)\}$. Thus we get, $  P((a,b),c)= Q(a,(b,c))$, $\forall ~a\in A$, $b\in B$ and $c\in C$. So,  $\{(S,A) \cap (G,B)\}\cap (H,C) = (S,A) \cap \{(G,B)\cap (H,C)\}$.\\

(c) (i) We know that, any two fuzzy sets $\mathcal{A}$ and $\mathcal{B}$ satisfy the De Morgan's law, i.e., $\overline{\mathcal{A} \cap  \mathcal{B}} = \overline{\mathcal{A}}\cup \overline{\mathcal{B}}$ and $\overline{\mathcal{A} \cup  \mathcal{B}} = \overline{\mathcal{A}}\cap \overline{\mathcal{B}}$.

Let $(S,A) \cap (G,B)=(P,A\times B)$ be a fuzzy soft set. Then, for any $(a,b)\in A\times B$, we have $P(a,b)=S(a)\cap G(b)$. Again, if we take complement of the fuzzy soft set $(P,A\times B)$, then the fuzzy soft set $C(P,A\times B)$ can be defined as $CP(a,b)=\overline{S(a)\cap G(b)}$. Thus we get $ C((S,A) \cap (G,B))=C(P,A\times B)=\{((a, b), \overline{S(a)\cap G(b)}): a\in A, b\in B\}$.

Similarly, if we take $C(S,A) \cup C(G,B)=(Q, A\times B)$ defined by $Q(a,b)=\overline{S(a)} \cup \overline{G(b)}$, then we get, $C(S,A) \cup C(G,B)=(Q, A\times B)=\{((a, b), \overline{S(a)}\cup \overline{G(b)}): a\in A, b\in B\}$.

 Since, $S(a)$ and $G(b)$ are fuzzy sets, so $\overline{S(a)\cap G(b)}=\overline{S(a)}\cup \overline{G(b)}$. Hence, we get $C((S,A) \cap (G,B)) = C(S,A) \cup C(G,B)$.\\

(d) (i) We know that the distribution property of union holds for any three fuzzy sets.

 Let $(S, A)\cup \{(G, B) \cap (H,C)\}=(P, A\times(B\times C))$ and $\{(S, A) \cup (G,B)\} \cap \{(S, A) \cup (H,C)\}=(Q,(A\times B)\times (A\times C))$. For any $(a, (b, c))\in A\times (B\times C)$, we define these fuzzy soft set as $P(a,(b,c))=S(a)\cup\{G(b)\cap H(c)\}$ and $Q((a,b),(a,c))=\{S(a)\cup G(b)\}\cap \{S(a)\cup H(c)\}$.

 But, in the soft set theory the two set of parameters $A\times(B\times C)$ and $(A\times B)\times (A\times C)$ are indistinguishable. Also for three fuzzy sets $S(a)$, $G(b)$ and $H(c)$ we get $S(a)\cup\{G(b)\cap H(c)\}=\{S(a)\cup G(b)\}\cap \{S(a)\cup H(c)\}$. Hence, $P(a,(b,c))=Q((a,b),(a,c))$, $\forall (a,b,c)\in A\times B\times C$. This implies that $(S,A)\cup \{(G,B) \cap (H,C)\} = \{(S,A) \cup (G,B)\} \cap \{(S,A) \cup (H,C)\}$.
    
\end{proof}

\begin{theorem}
 For any fuzzy non-empty soft set $(S,A)$ defined over a universal set $X$ with fuzzy null soft set $(\phi , -)$ and fuzzy absolute soft set $(X , -)$ over the same universal set $X$, the following properties hold: 

 (i) $C(\phi, -) \cong (X, -)$,
 
 (ii) $C(X, -) \cong (\phi, -)$,
 
 (iii) $(S, A) \cap (\phi, -) \cong (\phi, -)$,

 (iv) $(S, A) \cup (\phi, -) \cong (S, A)$,

 (v) $(S, A) \times (\phi, -) \cong (\phi, -)$,

 (vi) $(S, A) + (\phi, -) \cong (S, A)$,

 (vii) $(S, A) \cap (X, -) \cong (S, A)$,

 (viii) $(S, A) \cup (X, -) \cong (X, -)$,

 (ix) $(S, A) \times (X, -) \cong (S, A)$,

 (x) $(S, A) + (X, -) \cong (X, -)$.

\end{theorem}

\begin{proof}
 To prove the above results, we first consider 
$(S, A)=\{(a,S(a)): a\in A\}$, $(\phi, -)=\{(b,  \Tilde{\phi}): b\in B\}$ and $(X, -)=\{(c, \Tilde{X}): c\in C\}$, where $B$ and $C$ are the sets of parameters of $(\phi, -)$ and $(X, -)$ respectively. This implies that $\tau(S, A)=\{S(a):a \in A\}$, $\tau(\phi, -)=\{\Tilde{\phi} \}$ and $\tau(X, -)=\{\Tilde{X}\}$.\\

(i) The complement of the fuzzy null soft set $(\phi, -)$ is $C(\phi, -)=\{(b,  \overline{\Tilde{\phi}}): b\in B\}= \{(b, \Tilde{X}): b\in B\}$ and $\tau(C(\phi, -))=\{\Tilde{X}\}=\tau(X, -)$. This implies that $C(\phi, -) \cong (X, -)$.\\

(ii) We have, $C(X, -)=\{(c,  \overline{\Tilde{X}}): c\in C\}= \{(c, \Tilde{\phi}): c\in C\}$. Therefore, $\tau(C(X, -))=\{\Tilde{\phi}\}=\tau(\phi, -)$. Hence, $C(X, -) \cong (\phi, -)$. \\

(iii) We have, $(S, A)\cap (\phi, -)=\{((a, b), S(a)\cap \Tilde{\phi}) : a\in A, b\in B\}=\{((a,b), \Tilde{\phi}): a\in A, b\in B\}$. So, we get $\tau((S, A)\cap (\phi, -))=\{\Tilde{\phi} \}=\tau(\phi, -)$. Hence, $(S, A) \cap (\phi, -) \cong (\phi, -)$.\\

(iv) We have, $(S, A)\cup (\phi, -)=\{((a,b), S(a)\cup \Tilde{\phi}) : a\in A, b\in B\}=\{((a,b), S(a)): a\in A, b\in B\}$. Therefore, $\tau((S,A)\cup (\phi, -))=\{S(a): a\in A\}=\tau(S,A)$. Hence, $(S, A) \cup (\phi, -) \cong (S, A)$.\\

(v) We have, $(S,A) \times (\phi, -)=\{((a,b), S(a).\Tilde{\phi}) : a\in A, b\in B\}=\{((a,b), \Tilde{\phi}): a\in A, b\in B\} $. This implies that $\tau((S,A) \times (\phi, -))=\{\Tilde{\phi} \}=\tau(\phi, -)$. Hence, $(S, A) \times (\phi, -) \cong (\phi, -)$.\\

(vi) We have, $(S,A) + (\phi, -)=\{((a,b), S(a) \hat{+} \Tilde{\phi}) : a\in A, b\in B\}=\{((a,b), S(a)): a\in A, b\in B\}$. So, $\tau((S,A) + (\phi, -))=\{S(a): a\in A\}=\tau(S,A)$. Hence, $(S, A) + (\phi, -) \cong (S, A)$.\\

(vii) We have, $(S, A)\cap (X, -)=\{((a, c), S(a)\cap \Tilde{X}) : a\in A, c\in C\}=\{((a,c), S(a)): a\in A, b\in B\}$. Thus, we get $\tau((S, A)\cap (X, -))=\{S(a) : a\in A \}=\tau(S, A)$. Hence, $(S, A) \cap (X, -) \cong (S, A)$.\\

(viii) We have, $(S, A)\cup (X, -)=\{((a,c), S(a)\cup \Tilde{X}) : a\in A, c\in C\}=\{((a,c), \Tilde{X}): a\in A, c\in C\}$. Therefore, $\tau((S,A)\cup (X, -))=\{\Tilde{X}\}=\tau(X, -)$. Hence, $(S, A) \cup (X, -) \cong (X, -)$.\\

(ix) We have, $(S,A) \times (X, -)=\{((a,c), S(a).\Tilde{X}) : a\in A, c\in C\}=\{((a,c), S(a)): a\in A, c\in C\} $. This implies that $\tau((S,A) \times (X, -))=\{S(a) : a\in A \}=\tau(S, A)$. Hence, $(S, A) \times (X, -) \cong (S, A)$.\\

(x) We have, $(S,A) + (X, -)=\{((a,c), S(a) \hat{+} \Tilde{X}) : a\in A, c\in C\}=\{((a,c), \Tilde{X}): a\in A, c\in C\}$. It implies that, $\tau((S,A) + (X, -))=\{\Tilde{X}\}=\tau(X, -)$. Hence, $(S, A) + (X, -) \cong (X, -)$.\\
\end{proof}
\begin{theorem}
     Let $(S,A)$ and $(G,B)$ be two fuzzy soft sets defined over a universal set $X$ such that $(S,A)$ and $(G,B)$ are not fuzzy empty soft set, fuzzy null soft set and fuzzy absolute soft set. Again, let $(U,A\times B)$ and $(I,A\times B)$ be the fuzzy soft union and fuzzy soft intersection of $(S,A)$ and $(G,B)$ respectively. Then, the following properties hold:\\

    (i) if $S(a)\cup G(b)\in MAX(\tau(U,A\times B))$, for $(a, b)\in A\times B$, then $S(a)\in MAX(\tau(S,A))$ or $G(b)\in MAX(\tau(G,B))$,\\

    (ii) if $S(a)\cap G(b)\in MAX(\tau(I,A\times B))$, for $(a, b)\in A\times B$, then $S(a)\in MAX(\tau(S,A))$ or $G(b)\in MAX(\tau(G,B))$,\\
    
    (iii) if $S(a)\cap G(b)\in MIN(\tau(I,A\times B))$, for $(a, b)\in A\times B$, then $S(a)\in MIN(\tau(S,A))$ or $G(b)\in MIN(\tau(G,B))$,\\

    (iv) if $S(a)\cup G(b)\in MIN(\tau(U,A\times B))$, for $(a, b)\in A\times B$, then $S(a)\in MIN(\tau(S,A))$ or $G(b)\in MIN(\tau(G,B))$.

\end{theorem}
\begin{proof}

 We only proof (i) and (iii). \\

(i) Let $S(a)\cup G(b)(\neq \Tilde{X}) \in MAX(\tau(U,A\times B))$ for $(a, b)\in A\times B$. This implies that there does not exist any $S(a_i)\cup G(b_i)(\neq \Tilde{X})\in \tau(U,A\times B)$, for $(a_i, b_i)\in A\times B$ such that $S(a_i)\cup G(b_i)\supset S(a)\cup G(b)$. Suppose the contrary, $S(a) \notin MAX(\tau(S,A))$ and $G(b) \notin MAX(\tau(G,B))$. It implies $S(a)= \Tilde{X}$ or  $S(a_j)\supset S(a)(\neq \Tilde{X})$, for some $S(a_j) (\neq \Tilde{X}) \in \tau(S,A)$, and $G(b)= \Tilde{X}$ or $G(b_j)\supset G(b)(\neq \Tilde{X})$, for some  $G(b_j) (\neq \Tilde{X}) \in \tau(G,B)$. Now, if we consider $S(a)= \Tilde{X}$ and $G(b)= \Tilde{X}$ then $S(a) \cup G(b)= \Tilde{X} \cup \Tilde{X}=\Tilde{X}\notin MAX(\tau(U,A\times B))$, a contradiction. If $S(a)= \Tilde{X}$ and $G(b_j)\supset G(b)(\neq \Tilde{X})$ then $S(a) \cup G(b) = \Tilde{X} \cup G(b)= \Tilde{X} \notin MAX(\tau(U,A\times B))$, a contradiction. By similar process we can show that, if $S(a_j)\supset S(a)(\neq \Tilde{X})$ and $G(b)= \Tilde{X}$, then $S(a)\cup G(b)= \Tilde{X} \notin MAX(\tau(U,A\times B))$, which is again a contradiction. Finally, if we consider $S(a_j)\supset S(a)(\neq \Tilde{X})$ and $G(b_j)\supset G(b)(\neq \Tilde{X})$, it gives $S(a) \cup G(b)  \subset S(a_j)\cup G(b_j) \in \tau(I, A\times B)$, this means $S(a)\cup G(b) \notin MAX(\tau(U,A\times B))$, a contradiction. All these scenarios contradicts the inclusion $S(a)\cup G(b) \in MAX(\tau(U,A\times B))$. Hence, $S(a) \in MAX(\tau(S,A))$ or $G(b) \in MAX(\tau(G,B))$.\\

(iii) Let $S(a)\cap G(b)(\neq \Tilde{\phi})\in MIN(\tau(I,A\times B))$ for $(a, b)\in A\times B$. This implies that there does not exist any $S(a_k)\cap G(b_k)(\neq \Tilde{\phi})\in \tau(I,A\times B)$, for $(a_k, b_k)\in A\times B$ such that $S(a_k)\cap G(b_k)\subset S(a)\cap G(b)$. Now, let us assume that, $S(a) \notin MIN(\tau(S,A))$ and $G(b) \notin MIN(\tau(G,B))$, which implies $S(a)= \Tilde{\phi}$ or $S(a_n)\subset S(a)(\neq \Tilde{\phi})$, for some $S(a_n) (\neq \Tilde{\phi}) \in \tau(S,A)$, and $G(b)= \Tilde{\phi}$ or $G(b_n)\subset G(b)(\neq \Tilde{\phi})$, for some  $G(b_n) (\neq \Tilde{\phi}) \in \tau(G,B)$. Now, if $S(a)= \Tilde{\phi}$ and $G(b)= \Tilde{\phi}$ then $S(a) \cap G(b)= \Tilde{\phi} \cap \Tilde{\phi}=\Tilde{\phi}\notin MIN(\tau(I,A\times B))$, a contradiction. If $S(a)= \Tilde{\phi}$ and $G(b_n)\subset G(b)(\neq \Tilde{\phi})$ then $S(a) \cap G(b)= \Tilde{\phi} \cap G(b)= \Tilde{\phi}\notin MIN(\tau(I,A\times B))$, a contradiction. Using the same method, we can show that if $S(a_n)\supset S(a)(\neq \Tilde{\phi})$ and $G(b)= \Tilde{\phi}$ then $S(a)\cap G(b)\notin MIN(\tau(I,A\times B))$, which is again a contradiction. Finally, if we consider $S(a_n)\supset S(a)(\neq \Tilde{\phi})$ and $G(b_n)\supset G(b)(\neq \Tilde{\phi})$, it gives $S(a_n)\cap G(b_n) \supset S(a)\cap G(b) \in \tau(I, A\times B)$. Thus, $S(a)\cap G(b) \notin MIN(\tau(I,A\times B))$, a contradiction. Thus, based on these four scenarios, we can conclude that $S(a)\in MIN(\tau(S,A))$ or $G(b)\in MIN(\tau(G,B))$.\\
    
\end{proof}

\begin{example}

In general, the converse parts of the theorem 0.2 may not be true. We can consider the following example:\\

Consider the following two fuzzy soft sets $(S, A)$ and $(G, B)$ defined over a universal set $X$, where $X=\{a, b, c\}$, $A=\{x, y, z\}$, $B=\{p, q, r\}$,
$(S,A)=\{(x, \{\frac{a}{0.5}\}), (y, \{\frac{b}{0.4}\}), (z, \{\frac{c}{0.3}\})\}$ and $(G,B)=\{(p, \{\frac{b}{0.5}\}), (q, \{\frac{a}{0.4}\}), (r, \{\frac{c}{0.3}\})\}$.\\

In these two fuzzy soft sets, we get,
$\tau(S, A)=\{\{\frac{a}{0.5}\}, \{\frac{b}{0.4}\}, \{\frac{c}{0.3}\}\}$, $\tau(G, B)=\{\{\frac{a}{0.4}\}, \{\frac{b}{0.5}\}, \{\frac{c}{0.3}\}\}$, $MAX(\tau(S, A))= \{\{\frac{a}{0.5}\}, \{\frac{b}{0.4}\}, \{\frac{c}{0.3}\} \}$, $MIN(\tau(S, A))= \{\{\frac{a}{0.5}\}, \{\frac{b}{0.4}\}, \{\frac{c}{0.3}\} \}$, $MAX(\tau(G, B))= \{ \{\frac{b}{0.5}\}, \{\frac{a}{0.4}\}, \{\frac{c}{0.3}\} \}$, $MIN(\tau(G, B))= \{ \{\frac{b}{0.5}\}, \{\frac{a}{0.4}\}, \{\frac{c}{0.3}\} \}$.\\

Let $(U, A\times B)$ be the union of the fuzzy soft sets $(S,A)$ and $(G,B)$. Then,

$(U, A\times B)=\{((x, p), \{\frac{a}{0.5}, \frac{b}{o.5}\}), ((x, q), \{ \frac{a}{0.5}\}), ((x, r), \{\frac{a}{0.5}, \frac{c}{0.3}\}), ((y, p),\\ \{\frac{b}{0.5}\}), ((y, q), \{\frac{a}{0.4}, \frac{b}{0.4}\}), ((y, r), \{\frac{b}{0.4}, \frac{c}{0.3}\}), ((z, p), \{\frac{b}{0.5}, \frac{c}{0.3}\}), ((z, q), \{\frac{a}{0.4}, \frac{c}{0.3}\}), \\ ((z, r), \{\frac{c}{0.3}\}) $.\\

Now, we get the followings: \\

$\tau(U,A\times B)=\{\{\frac{a}{0.5}, \frac{b}{o.5}\}, \{ \frac{a}{0.5}\}, \{\frac{a}{0.5}, \frac{c}{0.3}\}, \{\frac{b}{0.5}\}, \{\frac{a}{0.4}, \frac{b}{0.4}\}, \{\frac{b}{0.4}, \frac{c}{0.3}\}, \{\frac{b}{0.5},\\ \frac{c}{0.3}\}, \{\frac{a}{0.4}, \frac{c}{0.3}\}, \{\frac{c}{0.3}\}\}$, $MAX(\tau(U,A\times B))=\{\{\frac{a}{0.5}, \frac{b}{o.5}\}, \{\frac{a}{0.5}, \frac{c}{0.3}\}, \{\frac{b}{0.5}, \frac{c}{0.3}\}\}$, $MIN(\tau(U,A\times B))=\{ \{\frac{a}{0.5}\}, \{\frac{b}{0.5}\}, \{\frac{c}{0.3}\} \}$.\\

From above, we can easily find that $S(x)=\{\frac{a}{0.5}\}\in MAX(\tau(S,A))$ and $G(q)=\{\frac{a}{0.4}\}\in MAX(\tau(G,B))$, but $U(x, q)=\{\frac{a}{0.5}\}\notin MAX(\tau(U,A\times B))$. Again, $S(x)=\{\frac{a}{0.5}\}\in MIN(\tau(S,A))$ and $G(p)=\{\frac{b}{0.5}\}\in MIN(\tau(G,B))$ but $U(x, p)=\{ \frac{a}{0.5}, \frac{b}{0.5}\}\notin MIN(\tau(U,A\times B))$.\\

Now, let $(I, A\times B)$ be the intersection of the fuzzy soft sets $(S,A)$ and $(G,B)$. Then, $(I, A\times B)=(S,A)\cap (G,B)=\{((x, p), \phi), ((x, q), \{ \frac{a}{0.4}\}), ((x, r), \phi), ((y, p),\\ \{\frac{b}{0.4}\}), ((y, q), \phi), ((y, r), \phi), ((z, p), \phi), ((z, q), \phi), ((z, r), \{\frac{c}{0.3}\}) $.\\

Here, we get the followings:\\

$\tau(I,A\times B)=\{\{ \frac{a}{0.4}\}, \{\frac{b}{0.4}\}, \{\frac{c}{0.3}\}\}$, $MAX(\tau(U,A\times B))=\{\{ \frac{a}{0.4}\}, \{\frac{b}{0.4}\}, \{\frac{c}{0.3}\}\}$, $MIN(\tau(U,A\times B))=\{\{ \frac{a}{0.4}\}, \{\frac{b}{0.4}\}, \{\frac{c}{0.3}\}\}$.\\

Now, we see that $S(x)=\{\frac{a}{0.5}\}\in MAX(\tau(S,A))$ and $G(p)=\{\frac{b}{0.5}\}\in MAX(\tau(G,B))$, but $I(x, p)=\phi \notin MAX(\tau(I,A\times B))$. Similarly $S(x)=\{\frac{a}{0.5}\}\in MIN(\tau(S,A))$ and $G(p)=\{\frac{b}{0.5}\}\in MIN(\tau(G,B))$, but $I(x, p)=\phi \notin MIN(\tau(I,A\times B))$.
\end{example}

\begin{theorem}
    Let $(S,A)$ be a fuzzy soft set defined over a universal set $X$ such that $S(a_i) \cap S(a_j)=\Tilde{\phi}$, $\forall a_i, a_j \in A$ and $a_i\neq a_j$ and let for each $a_i \in A$, $S(a_i)\neq \Tilde{X}$ and $S(a_i)\neq \Tilde{\phi}$. Then, each fuzzy set $S(a_i)\in \tau(S,A)$ is the maximal as well as minimal element of the family $\tau(S,A)$.
\end{theorem}

\begin{proof}
 Let us assume that $S(a_i)\notin MAX(\tau(S,A))$, for all $a_i\in A$. Then there exists $S(a_j)\in \tau(S,A)$ such that $S(a_j)\supset S(a_i)$. This implies that $S(a_i)\cap S(a_j)=S(a_i)\neq \Tilde{\phi}$, which contradicts our assumption of the fuzzy soft set $(S,A)$. Thus, we get $S(a_i)\in MAX(\tau(S,A))$, $\forall a_i\in A$. \\

Similarly, let $S(a_i)\notin MIN(\tau(S,A))$, for all $a_i\in A$. Then, there exists $S(a_j)\in \tau(S,A)$ such that $S(a_j)\subset S(a_i)$. It implies that $S(a_i)\cap S(a_j)=S(a_j)\neq \Tilde{\phi}$, which again contradicts our assumption of the fuzzy soft set $(S,A)$. Thus we get $S(a_i)\in MIN(\tau(S,A))$. Hence, each fuzzy set in the family $\tau(S,A)$ is both maximal and minimal of the family $\tau(S,A)$, $\forall a_i\in A$.\\

\end{proof}

\subsection{Matrix representations of fuzzy soft sets}

~~~~Matrix representations of fuzzy soft sets can be used in mathematical analysis, computations and comparisons of fuzzy soft sets. They allow for direct comparison between fuzzy soft sets to determine the similarity between different fuzzy soft sets. Also matrix offers a structured way to organize fuzzy soft set data. Here, we discuss the way in which fuzzy soft sets can be represented using matrices.\\

  A matrix $M$ that has the same number of rows as the cardinality of a universal set $X$ and number of columns as the cardinality of the set of parameters $A$ can be used to represent a fuzzy soft set $(S,A)$ defined over a universal set $X$. For the sake of practical computation and other real-world uses, we consider that the cardinalities of $X$ and $A$ are finite numbers. Now, let $X=\{x_1, x_2,\dots, x_n\}$ and $A=\{a_1, a_2, \dots, a_m\}$. Then, $|X|=n$ and $|A|=m$, where $n, m \in \mathbb{N}$. Then, the matrix representation table of fuzzy soft set $(S,A)$ can be formulated as follows:

  \begin{table}[h]
\begin{center}
\begin{tabular}{c|c c c c c}
& $S(a_1)$ & $S(a_2)$ & $S(a_3)$ & $\dots$ & $S(a_m)$\\
\hline
$x_1$ & $d_{11}$ & $d_{12}$ & $d_{13}$ & $\dots$ & $d_{1m}$\\
$x_2$ & $d_{21}$ & $d_{22}$ & $d_{23}$ & $\dots$ & $d_{2m}$\\
\vdots & \vdots & \vdots & \vdots & $\ddots$ & \vdots\\
$x_n$  & $d_{n1}$ & $d_{n2}$ &  $d_{n3}$ & \dots & $d_{nm}$

\end{tabular}
\end{center} 
\caption{\label{tab1}Matrix representation table of the fuzzy soft set $(S,A)$}
\end{table}

    Where $d_{ij}$ be the membership value of $x_i$ in the fuzzy set $S(a_j)$, for each $i=1,2,\dots, n$ and $j=1,2,\dots, m$. Therefore, each entry of the matrix representation of a fuzzy soft set is from the interval [0,1]. Hence, from the above representation, we get the following matrix $M$ of order $|X|\times |A|=n\times m$ as  matrix representation of the fuzzy soft set $(S,A)$. 
    
    \begin{center}
    \[
    M=
    \begin{bmatrix}
           d_{11} & d_{12} & d_{13} & \dots & d_{1m}\\
         d_{21} & d_{22} & d_{23} & \dots & d_{2m}\\
        \vdots & \vdots & \vdots & \ddots & \vdots\\
      d_{n1} & d_{n2} & d_{n3} & \dots & d_{nm}\\
     \end{bmatrix}
     \]
     \end{center}

     In this way we can represent every fuzzy soft set in its matrix form. Now, we consider some examples that explain the matrix representations of fuzzy soft sets.

  \begin{example}
      Let $(S,A)$ be a fuzzy soft set defined over a universal set $X$, where $X=\{x, y, z\}$ and $A=\{a, b, c\}$. We define the fuzzy soft set $(S,A)$ as follows:

     $S(a)=\{\frac{x}{{\alpha}_1}, \frac{y}{{\alpha}_2}, \frac{z}{{\alpha}_3}\}$, $S(b)=\{\frac{x}{{\beta}_1}, \frac{y}{{\beta}_2}, \frac{z}{{\beta}_3}\}$ and $S(c)=\{\frac{x}{{\gamma}_1}, \frac{y}{{\gamma}_2}, \frac{z}{{\gamma}_3}\}$, where each ${\alpha}_i, {\beta}_i, {\gamma}_i \in [0,1]$

     Then, we obtain the matrix $M$ for the fuzzy soft set $(S,A)$ as 

      \begin{center}
    \[
    M=
    \begin{bmatrix}
           {\alpha}_1 & {\beta}_1 & {\gamma}_1\\
         {\alpha}_2 & {\beta}_2 & {\gamma}_2\\
     {\alpha}_3 & {\beta}_3 & {\gamma}_3\\
     \end{bmatrix}
     \]
     \end{center}
    
  \end{example}

  \begin{remark}
 While using a matrix representation of a fuzzy soft set, it's important to note the order of the elements of the universal set and set of parameters. Altering the orders of the elements of the universal set and set of parameters of a fuzzy soft set may give different matrix representations for the same fuzzy soft set which makes it challenging to define matrix representations and related operations. Therefore, any fuzzy soft set having a fixed ordering of elements of the universal set and set of parameters, can be transformed into its unique matrix form. If the matrix form of a fuzzy soft set defined over a universal set is given, then we can easily identify the fuzzy soft set, where the elements of the universal set and set of parameters have some fixed orderings.\\

     From now on, we take the elements of the universal set and set of parameters of a fuzzy soft set will have some fixed order.
     
   \end{remark}

\subsection{Connection between fuzzy soft sets through matrix representations}

To explore the relationship between fuzzy soft sets using their matrices representations, we must first define some terms that will be utilized later. We already know that fuzzy soft set the results must be defined by the fuzzy set families $\tau(S,A)$ and $\tau(G,B)$ of two fuzzy soft sets, $(S,A)$ and $(G,B)$, respectively. Here, we only need to look at the columns of the matrix representation of a fuzzy soft set $(S,A)$, because each column $c_i$ represents the membership value of the elements of the fuzzy sets $S(a_i)$.

\begin{definition}
    For a matrix $M_{n\times m}$, we define the column set $\mathcal{C}$ as the classical set of all columns of the matrix $M_{n\times m}$, i.e. $\mathcal{C}=\{c_i : i=1, 2, \dots, m\}$, where $c_i$ is the $i^{th}$ column of the matrix $M_{n\times m}$.
\end{definition}

\begin{definition}
    Consider the column set $\mathcal{C}=\{c_i : i=1, 2, \dots, m\}$ of the matrix $M_{n\times m}=(d_{ij})_{n\times m}$. Then, for some column $c_i ,c_j \in \mathcal{C}$; we say that

        (i) $c_i = c_j$, if $d_{ki} = d_{kj}$ for each $k=1, 2, \dots, n$

        (ii) $c_i \geq c_j$, if $d_{ki} \geq d_{kj}$ for each $k=1, 2, \dots, n$ 
        
        (iii)$c_i > c_j$, if $d_{ki} \geq d_{kj}$ for each $k=1, 2, \dots, n$ and there exists at least one $k=1, 2, \dots, n$ such that $d_{ki} > d_{kj}$
        
        (iv) $c_i \leq c_j$, if $d_{ki} \leq d_{kj}$ for each $k=1, 2, \dots, n$ 
        
        (v) $c_i < c_j$, if $d_{ki} \leq d_{kj}$ for each $k=1, 2, \dots, n$ and there exists at least one $k=1, 2, \dots, n$ such that $d_{ki} < d_{kj}$
    
    where $d_{ij}$ be the $ij^{th}$ entry of the column $c_j$.
\end{definition}

\begin{definition}
Let $c_i$ and $c_j$ be two columns of a matrix $M_{n\times m}=(d_{ij})_{n\times m}$. Then, the maximum of these columns is again a column, say $c_k$, i.e., $c_k=max\{c_i, c_j\}$ and it is defined by $d'_{lk}= max\{d_{li}, d_{lj}\}$, where $d'_{lk}$, $d_{li}$ and $d_{lj}$ be the $l^{th}$ entry of the columns $c_i$, $c_j$ and $c_k$ respectively, $l=1,2,\dots, n$. Here the column $c_k$ may or may not be a column of the matrix $M_{n\times m}=(d_{ij})_{n\times m}$.
\end{definition}

\begin{definition}
Let $c_i$ and $c_j$ be two columns of a matrix $M_{n\times m}=(d_{ij})_{n\times m}$. Then, the minimum of these columns is again a column, say $c_k$, i.e., $c_k=min\{c_i, c_j\}$ and it is defined by $d'_{lk}= min\{d_{li}, d_{lj}\}$, where $d'_{lk}$, $d_{li}$ and $d_{lj}$ be the $l^{th}$ entry of the columns $c_i$, $c_j$ and $c_k$ respectively, $l=1,2,\dots, n$. Here the column $c_k$ may or may not be a column of the matrix $M_{n\times m}=(d_{ij})_{n\times m}$.
\end{definition}

Now, we are ready to discuss the relationships between fuzzy soft sets by using their matrix representations. Consider two fuzzy soft sets $(S,A)$ and $(G,B)$ defined over a universal set $X$ whose elements have some fixed-order. Let $M_{n\times m}$ and $M_{n\times p}$ are the matrices of the fuzzy soft sets $(S,A)$ and $(G,B)$ respectively.\\

\begin{theorem}

Let $M_{n\times m}$ be the matrix representation of the fuzzy soft set $(S,A)$ over the universal set $X$. Then the column set $\mathcal{C}$ of the matrix $M_{n\times m}$ corresponds the family of fuzzy sets $\tau (S,A)$ and conversely.

\end{theorem}
\begin{proof}

 Consider the matrix $M_{n\times m}$ to be the matrix representation of the fuzzy soft set $(S,A)$ defined over a universal set $X$ and let $\mathcal{C}$ be its column set. We already know that every column $c_i\in \mathcal{C}$ represents the membership value of the elements of the fuzzy sets $S(a_i)$, for some $i=1, 2,\dots, m$. Thus, we can say that $c_i$ corresponds the fuzzy set $S(a_i)$. We consider the following three cases:\\

\textbf{Case 1 :}
 If $|\mathcal{C}|=1$, i.e., all columns of the matrix $M_{n\times m}$ are equal then they corresponds to a single fuzzy set in $\tau(S,A)$. So, we get $|\tau(S,A)|=1$. Thus the column set $\mathcal{C}$ corresponds to the family $\tau(S,A)$.\\

 \textbf{Case 2 :}
If $|\mathcal{C}|=p$, for some integer $p<m$, then these $p$ columns correspond to $p$ different fuzzy sets in $\tau(S,A)$ and we get $|\tau(S,A)|=p$. So, the column set $\mathcal{C}$ corresponds to the family $\tau(S,A)$. \\

 \textbf{Case 3 :}
 If $|\mathcal{C}|=m$, i.e., all columns of the matrix $M_{n\times m}$ are distinct, then they correspond to $m$ different fuzzy sets in $\tau(S,A)$. Thus, we get $|\tau(S,A)|=m$. Hence the column set $\mathcal{C}$ corresponds to the family $\tau(S,A)$.\\

 The converse part of the above result is obvious.
    
\end{proof}

\begin{remark}
    Let $\mathcal{C}$ be the column set of the matrix representation $M_{n\times m}$ of a fuzzy soft set $(S,A)$ defined over a universal set $X$. If for any column $c_i \in \mathcal{C}$, every entry is equal to zero, then it corresponds to the fuzzy set $S(a)=\Tilde{\phi}$ for some $a\in A$ and we denote this column by $c_0$. Similarly, a column $c_i \in \mathcal{C}$ in which every entry is equal to one, corresponds to a fuzzy set $S(b)=\Tilde{X}$ for some $b\in A$ and it is denoted by $c_u$.
\end{remark}

 \begin{theorem}
     Two fuzzy soft sets $(S,A)$ and $(G,B)$ defined over a universal set $X$ are equal if and only if $M_{n\times m} = M_{n\times p}$ and $A=B$, where $M_{n\times m}$ and $M_{n\times p}$ are the matrix representations of the fuzzy soft sets $(S,A)$ and $(G,B)$ respectively. 
 \end{theorem}

 \begin{proof}
     
The proof of the first part is obvious. So, here we only prove the converse part of the theorem. Let $M_{n\times m} = M_{n\times p}$ and $A=B$. Since, $A=B$, then we get $|A|=|B|$. Thus the orders of the matrices $M_{n\times m}$ and $M_{n\times p}$ are equal and equal to $|X|\times |A|$. To prove $(S,A)=(G,B)$, we only have to show $S=G$. \\

We know that each column $c_i\in \mathcal{C}$ of the matrix representation of a fuzzy soft set $(S,A)$ corresponds to a fuzzy set $S(a_i)\in \tau(S,A)$. Let $c_i$ and $c'_i$ be the $i^{th}$ columns of the matrix $M_{n\times m}$ and $M_{n\times p}$ respectively. Since $M_{n\times m} = M_{n\times p}$, so $c_i = c'_i$, for all $i\in \mathbb{N}$ and $1\leq i \leq |A|=|B|$. Thus, the corresponding fuzzy sets $S(a_i)\in \tau(S,A)$ and $G(b_i)\in \tau(G,B)$ are equal, i.e., $S(a_i)=G(b_i)$ for all $i\in \mathbb{N}$ and $1\leq i \leq |A|=|B|$, $a_i \in A$ and $b_i \in B$. This implies that $S=G$. Hence, the fuzzy soft sets $(S,A)$ and $(G,B)$ are equal.\\

 \end{proof}

Let us consider an example to illustrate the above theorem.
 \begin{example}
     Let us consider $X=\{x, y, z\}$, $A=\{a, b\}$, $B=\{p,q\}$ and let $M_{3\times 2}$ and $N_{3\times 2}$ are the matrix representations of the fuzzy soft sets $(S,A)$ and $(G,B)$ respectively defined over a universal set $X$. 
     Again, let $M_{3\times 2} = N_{3\times 2}$ where
\begin{center}
    \[
    M_{3\times 2}=N_{3\times 2}=
    \begin{bmatrix}
         0.2 & 0.7 \\
          1 & 0.5 \\
            0.5 & 0 \\
     \end{bmatrix}
     \]
     \end{center}

     From the given matrix we get $S(a)=G(p)=\{\frac{x}{0.2}, \frac{y}{1}, \frac{z}{0.5}\}$ and $S(b)=G(q)=\{\frac{x}{0.7}, \frac{y}{0.5}\}$, this implies that $S=G$. In addition, if we let $A=B$ then we can say $(S,A)=(G,B)$.
     
 \end{example}

\begin{theorem}
 Let $M_{n\times m}$ and $M_{n\times p}$ be the matrix representations of the fuzzy soft sets $(S,A)$ and $(G,B)$ respectively defined over a universal set $X$. These two fuzzy soft sets $(S,A)$ and $(G,B)$ are said to be equivalent if and only if ${\mathcal{C}}_s={\mathcal{C}}_g$, where ${\mathcal{C}}_s$ and ${\mathcal{C}}_g$ are the column sets of the matrices $M_{n\times m}$ and $M_{n\times p}$ respectively.
\end{theorem}

\begin{proof}
    
 If the fuzzy soft sets $(S,A)$ and $(G,B)$ are equivalent, then we have $\tau(S,A)= \tau(G,B)$. Again from theorem 3.9, we know that the column set of the matrix representation of a fuzzy soft set $(S,A)$ corresponds to the family of set $\tau(S,A)$. Hence ${\mathcal{C}}_s={\mathcal{C}}_g$. Similarly, we can prove the other part of the theorem.\\
\end{proof}
We consider the following example to demonstrate the above theorem.

  \begin{example}
      Let $M_{3\times 2}$ and $N_{3\times3}$ be the matrix representations of the fuzzy soft sets $(S,A)$ and $(G,B)$ defined over a universal set $X$. We consider $X=\{x, y, z\}$, $A=\{a, b\}$ and $B=\{p, q, r\}$. Given that 
      \[
    M_{3\times 2}=
    \begin{bmatrix}
         0.2 & 0.7 \\
          1 & 0.5 \\
            0.5 & 0 \\
     \end{bmatrix}
     \] and 
     \[
    N_{3\times 3}=
    \begin{bmatrix}
        0.7 & 0.2 & 0.7 \\
          0.5 & 1 & 0.5 \\
           0 & 0.5 & 0 \\
     \end{bmatrix}
     \]

     From the given matrix $M_{3\times 2}$, we get ${\mathcal{C}}_s=\{c_a, c_b\}$, where $c_a$ and $c_b$ are the columns corresponding to the soft sets $S(a)$ and $S(b)$. Since $c_p = c_r$, so ${\mathcal{C}}_g=\{c_p, c_q\}$ is the column set of the matrix $N_{3\times3}$. Now, we see that $c_a = c_q$ and $c_b = c_p$. Hence, ${\mathcal{C}}_s={\mathcal{C}}_g$. Therefore, the fuzzy soft sets $(S,A)$ and $(G,B)$ are equivalent.
  \end{example}

\begin{theorem}

    Consider the matrix representation $M_{n\times m}$ of the fuzzy soft set $(S,A)$ defined over a universal set $X$. Then, \\
    
    (i) for any column $c_i(\neq c_0) \in \mathcal{C}$ of the matrix $M_{n\times m}$, if there does not exist $c_j(\neq c_i )\in \mathcal{C}$ such that $c_j < c_i$ then the fuzzy set $S(a)$, which corresponds the column $c_i$, belongs to $MIN(\tau (S,A))$,\\

     (ii) for any column $c_i(\neq c_u) \in \mathcal{C}$ of the matrix $M_{n\times m}$, if there does not exist $c_j(\neq c_i) \in \mathcal{C}$ such that $c_j > c_i$ then the fuzzy set $S(a)$, which corresponds the column $c_i$, belongs to $MAX(\tau (S,A))$.\\
\end{theorem}

\begin{proof}

(i) Let the matrix $M_{n\times m}$ be the matrix representation of the fuzzy soft sets $(S,A)$ defined over a universal set $X$. We already know that any column $c_i\in \mathcal{C}$ of the matrix $M_{n\times m}$ represents the fuzzy set $S(a_i)\in \tau(S,A)$, for $a_i\in A$, $i\in \mathbb{N}$ and $1\leq i \leq m$. Also, $c_0\in \mathcal{C}$ represent the empty fuzzy set in $\tau(S,A)$.

Now, let $c_i(\neq c_0) \in \mathcal{C}$ be any column of the matrix $M_{n\times m}$ and let there does not exist $c_j(\neq c_i )\in \mathcal{C}$ such that $c_j < c_i$. If the columns $c_i, c_j\in \mathcal{C}$ represent the fuzzy sets $S(a_i), S(a_j)\in \tau(S,A)$ respectively then the inequality $c_j < c_i$ implies that $S(a_i) \subset S(a_j)$. Thus we get, for $S(a_i) (\neq \tilde{\phi}) \in \tau(S,A)$, there does not exist $S(a_j)(\neq S(a_i)) \in \tau(S,A)$ such that  $S(a_i)\subset S(a_j)$. This implies that $S(a_i) \in MIN(\tau (S,A))$.

By the similar process we can prove the statement (ii).
\end{proof}

\begin{theorem}

     Let $M_{n\times m}$ and $M_{n\times p}$ be the matrix representations of the fuzzy soft sets $(S,A)$ and $(G,B)$ respectively defined over a universal set $X$. Consider the sets of parameters $A=\{a_1, a_2, \dots, a_m\}$ and $B=\{b_1, b_2, \dots, b_p\}$ and ${\mathcal{C}}_s=\{c_{a_i}: i=1, 2, \dots, m\}$ and ${\mathcal{C}}_g =\{c_{b_j}: j=1, 2, \dots, p\}$  be the column sets of the matrices $M_{n\times m}$ and $M_{n\times p}$ respectively, where $c_{a_i}$ and $c_{b_j}$ are the $i^{th}$ and $j^{th}$ column in the matrices $M_{n\times m}$ and $M_{n\times p}$ corresponding to the fuzzy sets $S(a_i)$ and $G(b_j)$ respectively. Then,\\

       (i) the soft set $(S,A)$ internally approximates the soft set $(G,B)$, written as $(S,A) \subseteq (G,B)$, if for any $c_{b_j} (\neq c_0) \in {\mathcal{C}}_g$, there exists $c_{a_i} (\neq c_0) \in {\mathcal{C}}_s$ such that $c_{a_i}\leq c_{b_j}$,\\

       (ii) the soft set $(S,A)$ externally approximates the soft set $(G,B)$, written as $(S,A) \supseteq (G,B)$, if for any $c_{b_j} (\neq c_u) \in {\mathcal{C}}_g$, there exists $c_{a_i} (\neq c_u) \in {\mathcal{C}}_s$ such that $c_{a_i}\geq c_{b_j}$.\\
       
\end{theorem}

\begin{proof}

(i) If the columns $c_{a_i}$ and $ c_{b_j}$ represent the fuzzy sets $S(a_i)$ and $G(b_j)$ in $\tau(S,A)$ and $\tau(G,B)$ respectively, then from the above statement we simply get, for $G(b_j) (\neq \tilde{\phi}) \in \tau(G,B)$, there exists $S(a_i) (\neq \tilde{\phi}) \in \tau(S,A)$ such that $S(a_i)\subseteq G(b_j)$. This implies that the soft set $(S,A)$ internally approximates the soft set $(G,B)$.\\

(ii) From the statement (ii) we get that, for $G(b_j) (\neq \Tilde{X}) \in \tau(G,B)$, there exists $S(a_i) (\neq \tilde{X}) \in \tau(S,A)$ such that $S(a_i)\supseteq G(b_j)$, as we know that the column $c_u$ represent the universal set $X$. Hence, we get $(S,A) \supseteq (G,B)$.
\end{proof}

  \subsection{Operations between fuzzy soft sets using matrix representations}

  In this subsection, we discuss unary and binary operations on fuzzy soft sets with the help of their matrix representation.

  \subsubsection{Fuzzy soft complement}

  Let $M_{n\times m}$, where $n,m \in \mathbb{N}$, be the matrix representation of the fuzzy soft set $(S,A)$ defined over a universal set $X$ and $C(S,A)$ be the complement of the fuzzy soft set $(S,A)$. Let $\mathcal{C}$ be the column set of the matrix $M_{n\times m}$. Then, the matrix $M'_{n\times m}$ with the column set $\mathcal{C'}$ is said to be the matrix representation of the fuzzy soft set $C(S,A)$ if $c'_i=1-{c_i}$, for all $i=1, 2, \dots, m$, $c_i \in \mathcal{C}$ and $c'_i \in \mathcal{C'}$.

  \subsubsection{Fuzzy soft intersection}

  Let $(S,A)$ and $(G,B)$ be two fuzzy soft sets defined over a universal set $X$. Also $M_{n\times m}$ and $N_{n\times p}$ be the matrix representations of these fuzzy soft sets with the column sets ${\mathcal{C}}_S$ and ${\mathcal{C}}_G$ respectively, where $n, m, p \in \mathbb{N}$. Let $(I, A\times B)=(S,A) \cap (G,B)$. Then, the matrix $P_{n\times mp}$ with column set ${\mathcal{C}}_I$, be the matrix representation of the fuzzy soft set $(I,A\times B)$ and it is defined as given below:
  \begin{center}
      $c_{ij}=min\{c_i, c_j\}$, for all $i=1, 2, \dots, m$ and $j=1, 2, \dots, p$
  \end{center}
where $c_i \in {\mathcal{C}}_S$, $c_j \in {\mathcal{C}}_G$ and $c_{ij} \in {\mathcal{C}}_I$.

   \subsubsection{Fuzzy soft union}

  Let $(S,A)$ and $(G,B)$ be two fuzzy soft sets defined over a universal set $X$. Assume that $M_{n\times m}$ and $N_{n\times p}$ be the matrix representations of these fuzzy soft sets with the column sets ${\mathcal{C}}_S$ and ${\mathcal{C}}_G$ respectively, where $n, m, p \in \mathbb{N}$. Let $(U, A\times B)=(S,A) \cup (G,B)$. Then, the matrix $P_{n\times mp}$ with column set ${\mathcal{C}}_I$, be the matrix representation of the fuzzy soft set $(I,A\times B)$ and it is defined as given below: 
  \begin{center}
      $c_{ij}=max\{c_i, c_j\}$, for all $i=1, 2, \dots, m$ and $j=1, 2, \dots, p$
  \end{center}
where $c_i \in {\mathcal{C}}_S$, $c_j \in {\mathcal{C}}_G$ and $c_{ij} \in {\mathcal{C}}_I$.\\

Now, we have several examples of matrix representations using the above mentioned operators.

  \begin{example}
      Let us consider the fuzzy soft set $(S,A)$ defined over a universal set $X$. Let $X=\{a, b, c\}$ and $A=\{x, y, z\}$. We define the fuzzy soft set as follows:
      
      $(S,A)=\{(x, \{\frac{a}{0.4}\}), (y, \{\frac{a}{0.5}, \frac{b}{0.2}, \frac{c}{0.3}\}), (z,\{\frac{a}{0.4}, \frac{b}{0.2}\})\}$\\

      Then, the matrix representation of the fuzzy soft set $(S,A)$ is given below:
\begin{center}
    \[
    M_{3\times3}=
    \begin{bmatrix}
        0.4 & 0.5 & 0.4 \\
          0 & 0.2 & 0.2 \\
           0 & 0.3 & 0 \\
     \end{bmatrix}
     \]
\end{center}
Again, let $N_{3\times 3}$ be the matrix representation of the fuzzy soft set $C(S,A)$. Now, we can easily form the matrix $N_{3\times 3}$. The matrix is shown below: 
\begin{center}
    \[
    N_{3\times 3}=
    \begin{bmatrix}
        0.6 & 0.5 & 0.6 \\
          1 & 0.8 & 0.8 \\
           1 & 0.7 & 1 \\
     \end{bmatrix}
     \]
\end{center}
  \end{example}

  \begin{example}
      Let us consider the matrix representations $M_{3\times 3}$ and $N_{3\times 2}$ of the fuzzy soft sets $(S,A)$ and $(G,B)$ respectively defined over a universal set $X$ respectively, as follows:
\[
    M_{3\times 3}=
    \begin{bmatrix}
        0.6 & 0.5 & 0.6 \\
          0.1 & 0.1 & 0.7 \\
           0.4 & 0 & 1 \\
     \end{bmatrix}
     \] and
     \[
    N_{3\times 2}=
    \begin{bmatrix}
        0.3 & 0.5 \\
          0 & 0.6 \\
          0.4 & 0.8 \\
     \end{bmatrix}
     \]
    
      Now, the matrix representation of the fuzzy soft set $(I,A\times B)=(S,A)\cap (G,B)$ is obtained as the matrix $I$ of order $3\times 6$ using 3.3.2, which is shown below:
 \[
    I_{3\times 6}=
    \begin{bmatrix}
        0.3 & 0.5 & 0.3 & 0.5 & 0.3 & 0.5\\
          0 & 0.1 & 0 & 0.1 & 0 & 0.6\\
          0.4 & 0.4 & 0 & 0 & 0.4 & 0.8\\
     \end{bmatrix}
     \]
      
      Similarly, the we can form the matrix representation $U$ of order $3\times 6$ for the fuzzy soft set  $(U,A\times B)=(S,A)\cup (G,B)$ using 3.3.3, which is shown below:

       \[
    U_{3\times 6}=
    \begin{bmatrix}
        0.6 & 0.6 & 0.5 & 0.5 & 0.6 & 0.6\\
          0.1 & 0.6 & 0.1 & 0.6 & 0.7 & 0.7\\
          0.4 & 0.8 & 0.4 & 0.8 & 1 & 1\\
     \end{bmatrix}
     \]
  \end{example}

\section{Decision-making in interviews using fuzzy soft set theory}

In this section, we present a hypothetical decision-making scenario in which two interviewers, Mr. Johan and Mr. Akim, are charged with selecting the best candidate for a position from three applicants, referred to as candidates A, B, and C, during an interview panel. The selection criteria are based on educational background $(x)$, experience $(y)$ and communication skills $(z)$. To make the selection process more structured and objective, the interviewers, Mr. Johan and Mr. Akim, decide to use fuzzy soft set theory, grading the candidates from 0 to 1. 

Mr. Johan and Mr. Akim construct two fuzzy soft sets $(S, A)$ and $(G, A)$ respectively, over a universal set $X$, utilizing their assessments of decision-making regarding all three candidates. They consider the set of parameters as $A=\{a, b, c\}$ and the universal set as $X=\{x, y, z\}$. Let these two fuzzy soft sets be defined as follows:\\ 
$S(a)=\{\frac{x}{0.9}, \frac{y}{0.8}, \frac{z}{0.8}\}$, $S(b)=\{\frac{x}{0.7}, \frac{y}{0.6}, \frac{z}{0.4}\}$ and $S(c)=\{\frac{x}{0.6}, \frac{y}{0.4}, \frac{z}{0.9}\}$, \\$G(a)=\{\frac{x}{0.7}, \frac{y}{0.8}, \frac{z}{0.9}\}$, $G(b)=\{\frac{x}{0.8}, \frac{y}{0.5}, \frac{z}{0.7}\}$ and $G(c)=\{\frac{x}{0.7}, \frac{y}{0.5}, \frac{z}{0.8}\}$.\\

Here, the membership values of $x$, $y$ and $z$ in each fuzzy set represent the points of $x$, $y$ and $z$ respectively, given by the interviewers, Mr. Johan and Mr.  Akim. \\

Now, to make a well-balanced decision, they need to use the intersection operation on these two fuzzy soft sets $(S, A)$ and $(G, A)$. Let $(H, A\times A)=(S, A) \cap (G, A)$. Thus, we obtain,\\

 $H(a, a)=\{\frac{x}{0.7}, \frac{y}{0.8}, \frac{z}{0.8}\}$, $H(a, b)=\{\frac{x}{0.8}, \frac{y}{0.5}, \frac{z}{0.7}\}$, $H(a, c)=\{\frac{x}{0.7}, \frac{y}{0.5}, \frac{z}{0.8}\}$, $H(b, a)=\{\frac{x}{0.7}, \frac{y}{0.6}, \frac{z}{0.4}\}$, $H(b, b)=\{\frac{x}{0.7}, \frac{y}{0.5}, \frac{z}{0.4}\}$, $H(b, c)=\{\frac{x}{0.7}, \frac{y}{0.5}, \frac{z}{0.4}\}$,
 $H(c, a)=\{\frac{x}{0.6}, \frac{y}{0.4}, \frac{z}{0.9}\}$,
 $H(c, b)=\{\frac{x}{0.6}, \frac{y}{0.4}, \frac{z}{0.7}\}$, and
 $H(c, c)=\{\frac{x}{0.6}, \frac{y}{0.4}, \frac{z}{0.8}\}$.\\
 
Here, the fuzzy soft set $(H, A\times A)$ represents the aggregate decision of Mr. Johan and Mr. Akim. Since the interviewers have to select only one candidate, they only need to consider the fuzzy sets $H(a, a)$, $H(b, b)$ and $H(c, c)$ of the fuzzy soft set $(H, A\times A)$. From these three fuzzy sets, it is evident that $H(b, b)\subset H(a, a)$ and $H(c, c)\subset H(a, a)$. This indicates that both Mr. Johan and Mr. Akim assign higher points to candidate A across all categories, including educational background, experience, and communication skills, in comparison to candidates B and C. Therefore, based on their assessments, it is likely that the interviewers will choose candidate A for the position.

\section{Conclusion}

In this paper, we have introduced and defined key concepts such as fuzzy soft set, equal fuzzy soft set, and equivalent fuzzy soft set, based on the correct notions of soft set introduced by Molodtsov \cite{1, 2, 3}. We have accurately described various unary and binary operations on fuzzy soft sets and developed significant results in this domain. Additionally, we presented fuzzy soft sets in matrix form and examined the behavior of two fuzzy soft sets through their matrix representations. This study included an exploration of unary and binary operations using their matrix representations. We believe that this paper will open a new chapter in the history of fuzzy soft set theory and aid researchers in developing further significant results in this area. Moving forward, we intend to further investigate and establish various ideas related to the correct notions of soft set theory as outlined by Molodtsov in our upcoming work.\\

{\bf Competing interests:} The authors declare that there is no competing interests.\\

{\bf Funding:} The authors did not receive any  funding.

\end{document}